\documentclass[12pt]{amsart}

\usepackage{amssymb}
\usepackage{mathrsfs}
\usepackage{amscd}
\usepackage[all,cmtip]{xy}

\newtheorem{theorem}{Theorem}[section]
\newtheorem{lemma}[theorem]{Lemma}
\newtheorem{proposition}[theorem]{Proposition}
\newtheorem{corollary}[theorem]{Corollary}

\theoremstyle{definition}
\newtheorem{definition}[theorem]{Definition}
\newtheorem{remark}[theorem]{Remark}
\newtheorem{example}[theorem]{Example}

\numberwithin{equation}{section}

\newcommand{\supp}{\mathrm{supp}\,}
\newcommand{\Autz}{\mathrm{Aut^\circ}\!}
\newcommand{\Pic}{\mathrm{Pic}}
\newcommand{\id}{\mathrm{id}}
\newcommand{\Span}{\mathrm{span}}
\newcommand{\Hom}{\mathrm{Hom}}
\newcommand{\Lie}{\mathrm{Lie}\,}
\renewcommand{\div}{\mathrm{div}}

\newcommand{\SL}{\mathrm{SL}}
\newcommand{\PGL}{\mathrm{PGL}}
\newcommand{\PSL}{\mathrm{PSL}}
\newcommand{\Spin}{\mathrm{Spin}}
\newcommand{\GL}{\mathrm{GL}}
\newcommand{\PSO}{\mathrm{PSO}}
\newcommand{\SO}{\mathrm{SO}}
\newcommand{\Bl}{\mathrm{Bl}}

\newcommand{\CC}{\mathbb C}
\newcommand{\QQ}{\mathbb Q}
\newcommand{\Chi}{\mathcal X}

\newcommand{\V}{\mathrm V}
\newcommand{\Z}{\mathrm Z}
\newcommand{\Zi}{\mathcal Z}
\newcommand{\F}{\mathcal F}

\newcommand{\ZZ}{\mathbb Z}
\newcommand{\PP}{\mathbb P}

\newcommand{\GG}{\mathbb G}
\newcommand{\XX}{\mathbb X}

\newcommand{\D}{\mathcal D}
\newcommand{\DD}{\mathbb D}

\newcommand{\E}{\mathcal E}
\newcommand{\EE}{\mathbb E}

\renewcommand{\P}{\mathrm P}
\newcommand{\eS}{\mathrm S}
\newcommand{\N}{\mathrm N}
\newcommand{\OO}{\mathcal O}
\newcommand{\T}{\mathcal T}

\begin{document}

\title{On reductive automorphism groups of regular embeddings}

\author{Guido Pezzini}
\address{Departement Mathematik \\
Friedrich-Alexander Universit\"at Erlangen-N\"urnberg\\
Cauerstra\ss e 11 \\ 
91058 Erlangen\\
Deutschland}
\email{pezzini@math.fau.de}


\begin{abstract}
Let $G$ be a connected reductive complex algebraic group acting on a smooth complete complex algebraic variety $X$. We assume that $X$ is a {\em regular embedding}, a condition satisfied in particular by smooth toric varieties and flag varieties. For any set $\D$ of $G$-stable prime divisors, we study the action on $X$ of the group $\Autz(X, \D)$, the connected automorphism group of $X$ stabilizing all elements of $\D$. We determine a Levi subgroup $A(X,D)$ of $\Autz(X,\D)$, and also relevant invariants of $X$ as a spherical $A(X,\D)$-variety. As a byproduct, we obtain a complete description of the inclusion relation between closures of $A(X,D)$-orbits on $X$.
\end{abstract}

\maketitle

\section{Introduction}
In the 1970's Michel Demazure described the connected automorphism groups of two distinguished classes of algebraic varieties equipped with the action of a connected reductive group $G$: complete homogeneous spaces (see \cite{De77}), and smooth complete toric varieties (see \cite{De70}).

These two classes of $G$-varieties admit a common generalization: the {\em regular embeddings}, here also called $G$-regular embeddings, defined independently in \cite{BDP90} and \cite{Gi89}. With the additional assumption of completeness, Fr\'ed\'eric Bien and Michel Brion showed that these varieties form a relevant class of {\em spherical varieties}, which are by definition normal $G$-varieties with a dense orbit of a Borel subgroup of $G$.

The goal of this paper is to study the connected automorphism group $\Autz(X)$ of a complete regular embedding $X$. More precisely, we are interested in the group $\Autz(X,\D)$, where $\D$ is any set of $G$-stable prime divisors of $X$, and $\Autz(X,\D)$ is the connected automorphism group of $X$ stabilizing all elements of $\D$.

Our results are divided in several steps according to additional hypotheses on $\D$ and on $G$, and in each step we accomplish two goals. The first one is to describe a Levi subgroup $A(X,\D)$ of $\Autz(X,\D)$, based on the knowledge of discrete invariants associated with the $G$-action of $X$. These invariants come from the theory of spherical varieties, and are: the group $\Lambda_G(X)$ of $B$-eigenvalues of $B$-eigenvectors of $\CC(X)$, the set of {\em spherical roots} $\Sigma_G(X)$ of $X$ (see Definition~\ref{def:sphroots}), the set $\Delta_G(X)$ of $B$-stable but not $G$-stable prime divisors of $X$, and the stabilizer $\P_G(X)$ of the open $B$-orbit of $X$. The last invariant is the {\em fan} $\F_G(X)$, a collection of strictly convex polyhedral cones in the vector space $\Hom_\ZZ(\Lambda_G(X),\QQ)$ (see Definition~\ref{def:fan}). The fan determines $X$ uniquely among the complete regular embeddings having the same open $G$-orbit, generalizing the fan associated with a toric variety. It also provides a combinatorial description of the $G$-orbits of $X$ and of the inclusion relation between $G$-orbit closures (see \cite{Kn91}).

Now $X$ is a spherical variety also under the action of $A(X,\D)$, and our second goal is to determine the above invariants of $X$ with respect to the action of $A(X,\D)$. In particular, this provides a combinatorial description of the $A(X,\D)$-orbits on $X$.


Our approach is based on the analysis of a spherical variety $\XX$ canonically associated with $X$ and equipped with a canonical $G$-equivariant map $X\to \XX$. The variety $\XX$, defined in Section~\ref{s:wclosure}, is obtained from $X$ using a procedure called {\em wonderful closure}, which is closely related to the well-known construction of the {\em spherical closure} of a generic stabilizer of a spherical variety. We introduce the wonderful closure because it turns out to give much more direct informations on the automorphisms of $X$ than the spherical closure. On the other hand {\em wonderful varieties} such as $\XX$ (see Definition~\ref{def:wonderful}) play a central role in the theory of spherical varieties (see e.g.\ \cite{Lu01}), and their automorphism groups have been already studied in  \cite{Br07} and \cite{Pe09}. 

Our study of the group $\Autz(X,\D)$ proceeds by ``successive approximation'' with a sequence of subgroups that starts with elements very closely related to $\Autz(\XX)$. More precisely, we consider the filtration
\[
\Autz(X,\partial X) \subseteq \Autz(X,\D \cup (\partial X)^\ell) \subseteq \Autz(X,\D),
\]
where we denote by $\partial X$ the set of all $G$-stable prime divisors of $X$ and by $(\partial X)^\ell$ the subset of $G$-invariant prime divisors mapping surjectively onto $\XX$.

The group $\Autz(X,\partial X)$ is also the connected group of automorphisms stabilizing all $G$-orbits, and is the subject of Section~\ref{s:allGorbits}. This group is reductive after \cite{Br07} and $X$ is regular under its action; moreover $\Autz(X,\partial X)$ is equal to $\Autz(\XX,\partial \XX)$ up to central isogeny and up to a torus factor. We recall that $\Autz(\XX,\DD)$ for any $\DD\subseteq \partial \XX$ is reductive, and known after \cite{Pe09}.


After providing some technical results relating the automorphisms of $X$ and $\XX$ in Sections~\ref{s:relating} and \ref{s:restricting}, we devote Section~\ref{s:nonlinear} to $\Autz(X,\D\cup (\partial X)^\ell)$. We show that this group is again reductive: it is equal to $\Autz(\XX, \DD)$ up to central isogeny and up to a torus factor, where $\DD\subseteq \partial\XX$ is now the set of the images $\pi(D)\in\partial \XX$ such that $D\in\partial X$ is stable under $\Autz(X,\D\cup (\partial X)^\ell)$.

This is the most technically involved part of the paper, and we use an indirect approach. First we analyze how the invariants of $X$ behave whenever $\Autz(X,\D\cup (\partial X)^\ell)$ is strictly bigger than $\Autz(X,\partial X)$. Under this hypothesis the fan $\F_G(X)$ has a particularly nice structure and is essentially determined by the subfan $\F_G^\Lambda(X)\subseteq \F_G(X)$ obtained intersecting all cones with a linear subspace naturally associated with $\DD$ (see Lemma~\ref{lemma:movedareorthogonal}).

Then, using the open $\Autz(\XX,\DD)$-orbit of $\XX$ and the subfan $\F_G^\Lambda(X)$, we build ex novo a regular variety $X_A$ under the action of a reductive group $A$ which is equal up to a torus factor to the universal cover of $\Autz(\XX,\DD)$ (see Proposition~\ref{proposition:H_A} ad Corollary~\ref{cor:fan}). We show in Proposition~\ref{proposition:H_A} and Theorem~\ref{thm:nonlinear:lifts} that $A$ is equal to $\Autz(X,\D\cup (\partial X)^\ell)$ up to central isogeny, and that $X_A$ is $A$-equivariantly isomorphic to $X$.

We study the group $\Autz(X,\D)$ in Section~\ref{s:abelian} for any $\D$ assuming that $G$ is abelian, so $X$ is a smooth complete toric variety. Here $\XX$ is least useful, because in this case it is a single point. However, the classical results by Demazure describe a Levi subgroup $A(X,\D)$ of $\Autz(X,\D)$, and it only remains to compute the invariants of $X$ as a spherical variety under its action (see Proposition~\ref{prop:abelian} and Theorem~\ref{thm:abelian}). In particular $X$ is not necessarily regular under the action of $A(X,\D)$, but we show that it is {\em horospherical}, i.e.\ it has no spherical root.

Section~\ref{s:semisimple} deals with the case where $\D$ is again arbitrary but $G$ is semisimple. Our main result here is that $\Autz(X,\D \cup (\partial X)^\ell)$ is itself a Levi subgroup of $\Autz(X,\D)$ (see Theorem~\ref{thm:semisimple}). Therefore, in this case, our results of Section~\ref{s:nonlinear} accomplish our goals in full generality.

Finally, we discuss in Section~\ref{s:linear} the group $\Autz(X,\D)$ with no restriction on $\D$ or $G$. Here we build on our results on $\Autz(X,\D\cup (\partial X)^\ell)$ together with our results on the abelian case, which can be applied to a general fiber $Y$ of the map $\pi\colon X\to \XX$ since $Y$ is a union of toric varieties.

We show in Propositions~\ref{prop:Levinl} and \ref{prop:Levinldesc}, that a Levi subgroup $A(X,\D)$ is equal, up to isogeny, to the product of the commutator of $\Autz(X,\D\cup (\partial X)^\ell)$ with a reductive subgroup of $\Autz(X')$, where $X'$ is a connected component of $Y$. We also compute the invariants of $X$ under the action of $A(X,\D)$ in Proposition~\ref{prop:linear} and Theorem~\ref{thm:linear}.

\subsection*{Acknowledgements}
I thank Jacopo Gandini for stimulating discussions, and I am especially grateful to Michel Brion for discussions and support during the development of this work. I thank the anonymous referees for useful remarks and suggestions on a previous version of this paper.

\subsection*{Notations}
Through this paper $G$ is a connected reductive linear algebraic group over the field of complex numbers $\CC$. We assume that $G = G'\times C$ where $C$ is an algebraic torus and $G'$ is semisimple and simply connected. We denote by $\GG_m$ the one-dimensional torus.

We fix a Borel subgroup $B\subseteq G$ and a maximal torus $T\subseteq B$. We denote by $B^-$ the Borel subgroup of $G$ such that $B\cap B^-=T$. If $\alpha$ is a root of $G$, then $\alpha^\vee$ denotes the corresponding coroot.

If $H$ is any algebraic group then we denote by $\Z(H)$ its center, by $H^\circ$ its connected component containing the unit element $e_H$, by $H^r$ its radical, by $H^u$ its unipotent radical, and by $\Chi(H)$ the set of its characters, i.e.\ algebraic group homomorphisms $H\to\GG_m$. If $V$ is an $H$-module, then we denote by $V^{(H)}$ the set of non-zero $H$-semiinvariants of $V$, and for any $\chi\in\Chi(H)$ we set%
\footnote{We underline that in our notation $V^{(H)}$ does {\em not} contain $0$ and the same holds for $V^{(B)}_\chi$, in contrast with the similar common notation $V_\chi$ which we don't use here.}
\[
V^{(H)}_\chi = \left\{ v\in V\smallsetminus\left\{0\right\} \;\middle\vert\; hv = \chi(h)v \;\forall h\in H \right\}.
\]
If $H\subseteq K$ are subgroups of $G$, then we denote by $\pi^{H,K}\colon G/H\to G/K$ the natural map sending $gH\in G/H$ to $gK\in G/K$, and by $\N_KH$ the normalizer of $H$ in $K$.

For any subset $R$ of a $\ZZ$-module $\Lambda$, we denote by $R^{\geq0}$ (resp.\ $R^\perp$) the subset of $\Hom_\ZZ(\Lambda,\QQ)$ of all elements that are $\geq 0$ (resp.\ $=0$) on $R$. We define in the same way {\em mutatis mutandis} the subsets $R^{\geq0},R^\perp\subseteq \Lambda$ for $R\subseteq \Hom_\ZZ(\Lambda,\QQ)$. The notation $\langle -, - \rangle$ denotes the natural pairing between $\Lambda$ and $\Hom_\ZZ(\Lambda,\QQ)$.

If $X$ is an algebraic variety, the connected component containing $\id_X$ of its automorphism group is denoted by $\Autz(X)$. If a connected algebraic group $H$ acts on $X$, we denote by
\[
\theta_{H,X}\colon H\to\Autz(X)
\]
the corresponding homomorphism.

If $X$ is a $G$-variety, we denote by $\Pic^G(X)$ the group of isomorphism classes of $G$-linearized invertible sheaves. If $X$ is normal and $Y$ is a Cartier divisor, then the invertible sheaf $\OO_X(Y)$ admits a (non unique) $G$-linearization (see \cite[Remark after Proposition~2.4]{KKLV89}). Through the paper $X$ will be in addition smooth, complete and quasi-homogeneous, and we will consider $G$-stable prime divisors. Then we may and will always assume that the $G$-linearization is chosen in such a way that the induced $G$-action on $H^0(X,\OO_X(Y))$ is equal to the action induced by the inclusion $H^0(X,\OO_X(Y))\subseteq\CC(X)$.

\section{Complete regular embeddings}\label{s:definitions}
\begin{definition}
Let $X$ be an irreducible $G$-variety with an open $G$-orbit. Then $X$ is a {\em $G$-regular embedding} if for any $x\in X$:
\begin{enumerate}
\item the closure $\overline{Gx}$ of its orbit is smooth, and if it is of positive codimension then it is the transversal intersection of the $G$-stable prime divisors containing it;
\item the stabilizer $G_x$ has a dense orbit on the normal space in $X$ to the orbit $Gx$ in the point $x$.
\end{enumerate}
\end{definition}
A $G$-regular embedding is smooth and has a finite number of $G$-orbits. Examples of $G$-regular embeddings are the $G$-homogeneous spaces for any $G$, and if $G$ is an algebraic torus then any smooth toric $G$-variety. Regular varieties occur naturally also in the theory of {\em spherical varieties}, i.e.\ normal $G$-varieties with a dense $B$-orbit.

More precisely, suppose that a $G$-variety $X$ is smooth and complete. Then $X$ is $G$-regular if and only if it is spherical and {\em toroidal}, i.e.\ any $B$-stable prime divisor containing a $G$-orbit is also $G$-stable (see \cite[Proposition 2.2.1]{BB96}).

We review some relevant invariants associated to any spherical $G$-variety $X$. If $x_0$ is a point on the open $G$-orbit of $X$, then we also denote the orbit $G x_0$ simply by $G/H$, where $H=G_{x_0}$ is called a {\em generic stabilizer} of $X$. In this case, $H$ is also called a {\em spherical subgroup}, and $(X,x_0)$ (or simply $X$) is called an {\em embedding} of $G/H$. A {\em morphism} between two embeddings $(X,x_0)$ and $(X',x_0')$ is a $G$-equivariant map $X\to X'$ sending $x_0$ to $x_0'$.

We will always assume that $x_0$ is chosen in such a way that $Bx_0$ is dense in $X$. Then $H$ is also called a {\em $B$-spherical subgroup}.

\begin{definition}\label{def:invariants}
Let $X$ be a spherical $G$-variety with open $G$-orbit $G/H$.
\begin{enumerate}
\item We define%
\footnote{We ignore the dependence on $B$ of all the invariants we define. This is justified by the fact that for any reductive group under consideration the choice of a Borel subgroup will be either unique (when the group is abelian) or always explicitly fixed.}
the lattice
\[
\Lambda_G(X) = \left\{\chi\in\Chi(B)\,\middle\vert\, \CC(X)^{(B)}_{\chi}\neq \varnothing\right\},
\]
whose rank is by definition the {\em rank} of $X$.
\item We define
\[
\N_G(X) = \Hom_\ZZ(\Lambda_G(X),\QQ).
\]
\item We define $\Delta_G(X)$ to be the set of {\em colors} of $X$, i.e.\ the $B$-stable prime divisors of $X$ having non-empty intersection with the open $G$-orbit of $X$.
\item\label{def:invariants:rho} For any discrete valuation $\nu\colon\CC(X)\smallsetminus\{0\}\to\QQ$ we define an element $\rho_{G,X}(\nu)\in\N_G(X)$ with the formula
\[
\langle \rho_{G,X}(\nu),\chi\rangle = \nu(f_\chi),
\]
where $f_\chi\in\CC(X)^{(B)}_\chi$. If $D$ is a prime divisor of $X$ and $\nu_D$ is the associated discrete valuation, then we will also write $\rho_{G,X}(D)$ for $\rho_{G,X}(\nu_D)$.
\item We define
\[
\V_G(X) = \left\{ \rho_{G,X}(\nu) \;|\; \nu \textrm{ is $G$-invariant} \right\},
\]
which is a polyhedral convex cone of maximal dimension in $\N_G(X)$ (see \cite{Br90}); we denote its linear part by $\V_G^\ell(X)$.
\item  We define the {\em boundary} of $X$, denoted by $\partial_G X$, to be the set of the irreducible components of $X\smallsetminus (G/H)$.
\item We define $\P_G(X)$ to be the stabilizer of the open $B$-orbit of $X$. The set of simple roots associated to $\P_G(X)$ is denoted by $\eS^p_G(X)$.
\end{enumerate}
\end{definition} 

For the above, and for all the invariants defined later, we will drop the indices $G$ or $X$ whenever it is clear which group or which variety are considered. In loose terms the colors of $X$ can also be considered as invariants under $G$-equivariant birational maps, since they are the closures in $X$ of the colors of $G/H$.

The Luna-Vust theory of embeddings of homogeneous spaces specializes for spherical toroidal varieties in the following way (for details and proofs see \cite{Kn96}).

\begin{definition}\label{def:fan}
Let $X$ be a toroidal spherical $G$-variety, and $Y$ an irreducible $G$-stable locally closed subvariety. Then we define  $c_{X,Y}\subseteq \N(X)$ to be the polyhedral convex cone generated by $\rho(D_1),\ldots,\rho(D_n)$, where  $D_1,\ldots,D_n$ are the $B$-stable prime divisors containing $Y$. The {\em fan} of $X$ is defined as
\[
\F_G(X)=\left\{c_{X,Y} \;\middle\vert\; Y\mbox{ a $G$-orbit of $X$}\right\}.
\]
\end{definition}

Notice that since $X$ is toroidal then the divisors $D_1$,\ldots,$D_n$ above are also $G$-stable for any $Y$. The collection of convex cones $\F(X)$ satisfies the following properties:

\begin{enumerate}
\item \label{fan1} each cone of $\F(X)$ is contained in $\V(G/H)$, it is strictly convex, and all its faces belong to $\F(X)$,
\item \label{fan2} any element of $\V(G/H)$ belongs to the relative interior of at most one cone of $\F(X)$.
\end{enumerate}
The map $X\mapsto \F(X)$ induces a bijection between toroidal embeddings of $G/H$ (up to isomorphism of embeddings) and {\em fans}, i.e.\ collections of strictly convex polyhedral convex cones satisfying (\ref{fan1}) and (\ref{fan2}).

The {\em support} of a fan $\F$ is defined as
\[
\supp\F = \bigcup_{c\in\F} c.
\]

The variety $X$ is complete if and only if $\supp\F(X)=\V(X)$, and it is smooth if and only if for each $c\in\F(X)$ there exists a basis $\gamma_1\ldots,\gamma_r$ of $\Lambda(X)$ and an integer $k$ between $1$ and $r$ such that
\[
c = \left\{\gamma_1,\ldots,\gamma_k\right\}^{\geq0}.
\]

For later reference, we recall that if a spherical embedding $X$ is not toroidal, then it is also described by a similar datum, called a fan of {\em colored convex cones}. Here, the convex cone associated to a $G$-orbit $Y\subseteq X$ is replaced by the pair $(c_{X,Y},d_{X,Y})$ where $d_{X,Y}$ is the set of colors containing $Y$, and $c_{X,Y}$ is defined as above.

In general, the set $\V(X)$ is also a polyhedral convex cone of maximal dimension in $\N_G(X)$, and its linear part $\V^\ell(X)$ has dimension (as a $\QQ$-vector space) equal to the dimension of $\N_GH/H$ (as a complex algebraic group). The equations defining the maximal proper faces of $\V(X)$ are linearly independent (see \cite[Corollaire 3.3]{Br90}). In other words, there always exist $\sigma_1,\ldots,\sigma_k\in\Lambda(X)$ that are indivisible, linearly independent, and such that
\[
\V(X) = \left\{-\sigma_1,\ldots,-\sigma_k\right\}^{\geq0}.
\]

\begin{definition}\label{def:sphroots}
The elements $\sigma_1,\ldots,\sigma_k$ above are uniquely determined by $G/H$ and called the {\em spherical roots} of $X$; their set is denoted by $\Sigma_G(X)$.
\end{definition}

The map $Y\mapsto c_{X,Y}$ sends a $G$-orbit of codimension $d$ in $X$ to a cone of dimension $d$, and this restricts to a bijection between the boundary $\partial X$ and the set of $1$-dimensional cones in $\F(X)$.

Whether the cone of a given prime divisor $D\in\partial X$ lies or not on the linear part of the valuation cone has a strong influence on the automorphisms of $X$ not stabilizing $D$. This motivates the following definition.
\begin{definition}\label{def:ell}
For a subset $\D\subseteq \partial X$, we define the subsets
\[
\D^\ell = \left\{Y\in\D\,\middle\vert\, c_{X,Y}\subset\V^\ell(X) \right\}
\]
and
\[
\D^{n\ell} = \D\smallsetminus\D^\ell.
\]
\end{definition}

\begin{example}\label{ex:blowup0}
Consider $G=\SL(n+1)$ (with $n\geq 2$) acting linearly and diagonally on $\PP^{n+1}\times (\PP^n)^*$, where on the first factor it acts only on the first $n+1$ homogeneous coordinates. Choose $B$ to be the subgroup of upper triangular matrices. Then $X=\Bl_p(\PP^{n+1})\times (\PP^n)^*$, with $p=[0,\ldots,0,1]$, is a $G$-regular variety with three $G$-stable prime divisors. It has rank $2$, and $\Lambda_G(X)$ is generated by the first and the last fundamental dominant weights, denoted resp.\ $\omega^G_1$ and $\omega^G_n$. The variety $X$ has two colors, namely the inverse images of the $B$-stable hyperplanes of its two factors. With respect to the basis of $\N_G(X)\cong\QQ^2$ dual to $(\omega^G_1,\omega^G_n)$, the two colors have valuations $(1,0)$ and $(0,1)$, and the $G$-stable prime divisors $D_1$, $D_2$ and $E$ have valuations resp.\ $(-1,1)$, $(1,-1)$ and $(-1,0)$. The fan $\F_G(X)$ has two maximal cones, generated resp.\ by $\rho(D_1)$, $\rho(E)$ and by $\rho(E)$, $\rho(D_2)$, and $X$ has exactly one spherical root, namely $\sigma=\omega_1+\omega_n$. The divisor $E$ is the unique $G$-stable prime divisor not lying on the linear part of the valuation cone.
\end{example}

\section{Spherical and wonderful closure}\label{s:wclosure}
In this section we recall the notion, introduced in \cite{Lu01}, of the spherical closure $\overline H$ of a spherical subgroup $H\subseteq G$. We also define another subgroup containing $H$, called its wonderful closure. This is essentially already known, but not yet found in the literature.

Since we only consider the group $G$, in this section we drop all subscripts of the invariants of Definition~\ref{def:invariants}.

An element $n$ of the normalizer $\N_GH$ of $H$ induces a $G$-equivariant isomorphism $G/H\to G/H$ given by $gH\mapsto gnH$. This induces an action of $\N_GH$ on the set of colors $\Delta(G/H)$: the {\em spherical closure} $\overline H$ of $H$ is defined as the kernel of this action.

If $\overline H=H$ then $H$ is called {\em spherically closed}, and for any spherical subgroup $H\subseteq G$ the spherical closure $\overline H$ is itself spherically closed. This is well known, and also follows easily from \cite[Lemma 2.4.2]{BL11}; we provide here a direct proof.

\begin{proposition}
For any spherical subgroup $H\subseteq G$, the spherical closure $\overline H$ is spherically closed.
\end{proposition}
\begin{proof}
Since $\overline H$ is contained in $\N_G H$ the quotient $\overline H/H$ is diagonalizable (see \cite[Theorem 6.1]{Kn94}), and thus $H$ is defined inside $\overline H$ as intersection of kernels of some characters. The colors of $G/\overline H$ generate $\Pic^G(G/\overline H)$ (see \cite[Proposition 2.2]{Br89}) and the latter is isomorphic to $\Chi(\overline H)$ (see \cite[Section 3.1]{KKV89}), therefore $\overline{\overline H}$ acts trivially on $\Chi(\overline H)$.

This implies that $\overline{\overline H}$ normalizes $H$. By definition, it fixes all colors of $G/\overline H$, but these correspond to the colors of $G/H$ via the natural map $\pi^{H,\overline H}\colon G/H\to G/\overline H$. Hence $\overline{\overline H}\subseteq\overline H$.
\end{proof}

For later convenience we report the following auxiliary result. Recall that whenever $H\subseteq K$ are spherical subgroups of $G$, the lattice $\Lambda(G/K)$ is contained in the lattice $\Lambda(G/H)$, since $B$-semiinvariant functions can be lifted from $G/K$ to $G/H$ via the map $\pi^{H,K}\colon G/H\to G/K$. We also denote this inclusion as a map $(\pi^{H,K})^*\colon\Lambda(G/K)\to\Lambda(G/H)$, which induces a surjection $\pi^{H,K}_*\colon\N(G/H)\to\N(G/K)$. 
\begin{lemma}\label{lemma:inclusion}
Let $H\subseteq K\subseteq \overline H$ be spherical subgroups of $G$. Then
\[
(\pi^{H,K}_*)^{-1}(\V(G/H)) = \V(G/K).
\]
\end{lemma}
\begin{proof}
The claim stems from $\pi^{H,K}_*(\V(G/H)) = \V(G/K)$ together with $\ker\pi_*^{H,\overline H}=\V^\ell(G/H)$ (see \cite[Theorem 4.4 and Theorem 6.1]{Kn96}) and
\[
\ker\left(\pi^{H,K}_*\right) \subseteq \V^\ell(G/H),
\]
which follows from $\pi^{K,\overline H}_*\circ \pi^{H,K}_* =\pi_*^{H,\overline H}$.
\end{proof}

A class of subgroups slightly broader then the spherically closed ones is the following.

\begin{definition}\label{def:wonderful}
If $\Sigma(G/H)$ is a basis of $\Lambda(G/H)$ then $H$ is called a {\em wonderful} subgroup of $G$. In this case there exists a fan having only one maximal cone equal to $\V(G/K)$; the associated toroidal embedding is denoted by $\XX(G/H)$.
\end{definition}

If $H$ is wonderful then the embedding $\XX(G/H)$ is smooth, has a unique closed $G$-orbit and it is a {\em wonderful variety} in the sense of \cite{Lu01}. A fundamental theorem of Knop (see \cite[Corollary 7.6]{Kn96}) states that a spherically closed subgroup is wonderful.

\begin{example}
The converse of the above statement is false: for example, if $G=\SO(2n+1)$ with $n\geq 2$, then $H=\SO(2n)$ is a wonderful subgroup, with $\overline H = \N_{\SO(2n+1)}\SO(2n)\neq H$ (see \cite[Cases 7B, 8B of Table 1]{Wa96}).
\end{example}

It is possible to define canonically a minimal wonderful subgroup $\widehat H$ between $H$ and $\overline H$. As a byproduct, the automorphism groups of regular embeddings of $G/H$ are more directly related to the automorphism group of $\XX(G/\widehat H)$ than to that of $\XX(G/\overline H)$.

\begin{definition}
Let $H$ and $I$ be spherical subgroups of $G$. Then $I$ is a {\em wonderful closure} of $H$ if it is wonderful, satisfies $H\subseteq I \subseteq \overline H$, and is minimal with respect to these properties.
\end{definition}

\begin{example}
Any wonderful subgroup $H\subseteq G$ is its own wonderful closure. We give an example of $H$ having wonderful closure different from $H$ and from $\overline H$. Let $G$ be $\SO(9)$, and let $P$ be a parabolic subgroup of $G$ with Levi subgroup $L=\GL(2)\times\SO(5)$. Set $L_H=\SL(2)\times \SO(4)\subset L$, and let $H^u\subset P^u$ be the subgroup stable under conjugation by $L_H$ and such that $P^u/H^u$ is equivariantly isomorphic to $\CC^2\otimes \CC$ under the action of $L_H$. Then $H=L_HH^u$ is a spherical subgroup of $G$; it is not wonderful because it has infinite index in its normalizer. Consider now the subgroup $L_K=\GL(2)\times \SO(4)$, and $K=L_KH^u$. The group $K$ is minimal among the subgroups of $G$ strictly containing $H$ and of finite index in their normalizers. In addition $K$ is wonderful of rank $2$ (see \cite[First case 6 of Table 2]{Wa96}), so it is a wonderful closure of $H$. On the other hand the normalizer of $H$ is equal to $\overline H$ (see \cite[Second case 6 of Table 2]{Wa96}), and $\overline H$ contains $K$ strictly (indeed $[N_GH:K]=2$).
\end{example}

We will show that a wonderful closure always exists and is unique; for this we need to recall a combinatorial description due to D.\ Luna of all spherical subgroups having spherical closure equal to $\overline H$. This is based on the well known combinatorial notion of {\em augmentations} and is stated in \cite[Proposition~6.4]{Lu01}. However, here we build our discussion on some other auxiliary results from \cite[Section 6]{Lu01}: this approach will be useful in Section~\ref{s:nonlinear}.

Let us fix a spherically closed subgroup $K$, and consider the following diagram
\[
\begin{CD}
0 @>>>\Lambda(G/K) @>\overline\rho>>  \Pic^G(\XX(G/K)) @>\tau>> \Pic^G(G/K) @>>> 0 \\
@. @.                                 @VV{\sigma}V \\
@. @.                                 \Pic^G(G/B)
\end{CD}
\]
where the row is exact (see also \cite[Proposition~2.2.1]{Br07}).

The map $\tau$ is the pullback along the inclusion $G/K\to \XX(G/K)$. For $\sigma$, observe that $\XX(G/K)$ has a unique closed $G$-orbit $Z$, which is projective and therefore is equipped with a $G$-equivariant map $G/B\to Z$. The map $\sigma$ is then the pullback along the composition $G/B\to Z\to \XX(G/K)$.

The map $\overline\rho$ is defined in the following way: for any $\chi\in\Lambda(G/K)$ we pick a function $f_\chi\in\CC(G/K)^{(B)}_\chi$ and consider the $G$-stable part $D=\div(f_\chi)^G$ of $\div(f_\chi)$. Then we set $\overline\rho(\chi)=\OO_\XX(-D)$, with the unique $G$-linearization such that $C$ acts trivially on the total space of the bundle (it exists, since $C$ acts trivially on $\XX$).

These maps admit also a different interpretation, using the fact that $G=C\times G'$ and $K\supseteq C$, that $\Delta(G/K)$ is a basis of $\Pic(\XX(G/K))$ (see \cite[Proposition 2.2]{Br89}), and the isomorphisms $\Pic^G(G/K)\cong \Chi(K)$, $\Pic^G(G/B)\cong \Chi(B)$. The resulting diagram
\begin{equation} \label{eqn:diagram}
\begin{CD}
0 @>>> \Lambda(G/K) @>\overline\rho>>  \Chi(C)\times\ZZ^\Delta @>\tau>> \Chi(K)  @>>> 0\\
@. @.                                 @VV{\sigma}V \\
@. @.                                 \Chi(B)
\end{CD}
\end{equation}
where $\Delta = \Delta(G/K)$, is also described in details in \cite[Section 6.3]{Lu01}. As in \cite[Lemme 6.2.2]{Lu01}, the product $\Chi(C)\times\ZZ^\Delta$ is identified with the quotient
\[
\frac{\CC(G)^{(B\times K)}}{\CC^\times}
\]
where $B$ acts on $G$ by left translation, $K$ by right translation, and $\CC^\times$ is the multiplicative group of constant non-zero functions on $G$. Under this identification, the map $\tau$ associates to a class $[f]$ the $K$-eigenvalue of $f$, the map $\sigma$ its $B$-eigenvalue, and the map $\overline\rho$ associates to $\chi\in\Lambda(G/K)$ the class of a function in $\CC(G)^{(B\times K)}_{(\chi,0)}$. The composition $\sigma\circ\overline\rho$ is the identity on $\Lambda(G/K)$, and $\overline\rho$ can also be written as
\[
\overline\rho(\chi) = (\chi|_{C},\langle\rho_{G/K}(\cdot),\chi\rangle)
\]
(see {\em loc.cit.}).

\begin{lemma}\cite[Lemme 6.3.1, Lemme 6.3.3, Lemme 6.4.1]{Lu01}\label{lemma:lunalattice}
Let $K\subseteq G$ be a spherically closed subgroup.
\begin{enumerate}
\item The map
\[
H\to \tau^{-1}\left(\Chi(K)^H\right)
\]
is an inclusion-reversing bijection between the set of normal subgroups $H$ of $K$ such that $K/H$ is diagonalizable, and the set of subgroups of $\Chi(C)\times \ZZ^\Delta$ containing $\overline\rho(\Lambda(G/K))$.
\item If the restriction of $\sigma$ to $\tau^{-1}\left(\Chi(K)^H\right)$ is injective then $H$ is spherical.
\item If in addition $\sigma(\tau^{-1}\left(\Chi(K)^H\right))\cap \eS_G^\circ(G/K)=\varnothing$ then $\overline H=K$, where the set $\eS_G^\circ(G/K)$ is the set of all simple roots $\alpha\in \frac12\Sigma(G/K)$ such that $\langle\beta^\vee,\alpha\rangle$ is even for all simple roots $\beta\in \frac12\Sigma(G/K)$.
\end{enumerate}
\end{lemma}

Before showing existence and uniqueness of the wonderful closure, we provide some auxiliary results. They are actually already contained in \cite[Section 6]{Lu01}, we reprove them here for convenience.

\begin{lemma}\label{lemma:equation}
For any spherical subgroup $H\subseteq G$ contained and normal in another subgroup $K\subseteq G$, and all $D\in\Delta(G/H)$, we have $\pi^{H,K}(D)\in\Delta(G/K)$, and
\[
\pi_*^{H,K}(\rho_{G/H}(D))=\rho_{G/K}(\pi^{H,K}(D)).
\]
\end{lemma}
\begin{proof}
Since $H$ is normal in $K$, then $K$ stabilizes the open set $BH\subseteq G$ acting by right multiplication on $G$ (see \cite[First part of the proof of Proposition 5.1]{BP87}). The complement $G\smallsetminus BH=G\smallsetminus BK$ is the union of $(\pi^{\{e_G\},H})^{-1}(E)$ for $E$ varying in $\Delta(G/H)$, and also the union of $(\pi^{\{e_G\},H})^{-1}(F)$ for $F$ varying in $\Delta(G/K)$, whence the first statement.

For the second statement, it is enough to show that a local equation of $D$ on $G/H$ can be chosen to be the pull-back of a function on $G/K$ along $\pi^{H,K}$. Let $E_1,\ldots,E_n$ be all the distinct $B$-stable prime divisors of $G$ such that $\pi^{\{e_G\},K}(E_i)=\pi^{H,K}(D)$. Since $G$ is factorial we can choose a global equation $f_i\in\CC[G]$ for each $E_i$, and consider the product $f=f_1\cdot\ldots\cdot f_n$.

The divisor $\div(f)$ on $G$ is $B$-stable under the left translation action of $G$ on itself, but none of its components is $G$-stable therefore there exists an element $g\in G$ such that the function $f_0\colon x\mapsto  f(g x)$ doesn't vanish on any divisor $E_i$. On the other hand $\div(f)$ is $K$-stable under the right translation action of $G$ on itself, thus $f$ is $K$-semiinvariant under this action. The function $f_0$ is then also $K$-semiinvariant, with same $K$-eigenvalue of $f$. It follows that
\[
F = \frac{f}{f_0}
\]
is $K$-invariant with respect to the right translation action. In other words $F=(\pi^{\{e_G\}, K})^*(\widetilde F)$ for some $\widetilde F\in\CC(G/K)$.

Now for some $i_0$ the divisor $E_{i_0}$ satisfies $\pi^{\{e_G\},H}(E_{i_0})=D$. The function $F$ is equal to the pull-back of $(\pi^{H,K})^*(\widetilde F)$ along $\pi^{\{e_G\},H}$ and is a local equation of $E_{i_0}$ on $G$, hence $(\pi^{H,K})^*(\widetilde F)$ is a local equation of $\pi^{\{e_G\},H}(E_{i_0})=D$ on $G/H$: the lemma follows.
\end{proof}

If we suppose in the above lemma that $\overline H=K$ then $\Delta=\Delta(G/K)$ and $\Delta(G/H)$ are identified via the map $\pi^{H,K}$ compatibly with the maps $\rho_{G/H}$ and $\rho_{G/K}$. In this case we can extend the map $\overline\rho$ to $\Lambda(G/H)$ as follows.

\begin{definition}
For any $H$ and $K$ as in Lemma~\ref{lemma:equation}, with $\overline H=K$, we denote again by $\overline\rho$ the extension of the above map $\overline\rho\colon\Lambda(G/K)\to\Chi(C)\times\ZZ^\Delta$ to $\Lambda(G/H)$ defined by:
\[
\overline\rho(\chi) = (\chi|_{C},\langle\rho_{G/H}(\cdot),\chi\rangle).
\]
\end{definition}

\begin{lemma}
For any $H$ and $K$ as in Lemma~\ref{lemma:equation} satisfying $\overline H=K$, the map  $\overline\rho\colon\Lambda(G/H)\to\Chi(C)\times\ZZ^\Delta$ is injective and the composition $\sigma\circ\overline\rho$ is the identity on $\Lambda(G/H)$. Moreover
\begin{equation}\label{eqn:lambdachar}
\overline\rho(\Lambda(G/H))=\tau^{-1}\left(\Chi(K)^H\right).
\end{equation}
\end{lemma}
\begin{proof}
First we observe that 
\[
(\CC(G)^{(B)})^H \subseteq \CC(G)^{(B\times K)}.
\]
Indeed, let $f\in(\CC(G)^{(B)})^H$. Then $\div(f)$ is $B$-stable under left translation, and $H$-stable under right translation. But $\Delta(G/H)$ and $\Delta(G/K)$ are identified via $\pi^{H,K}$, which implies that $H$ and $K$ act on the set of irreducible components of $G\smallsetminus BH=G\smallsetminus BK$ with the same orbits. Therefore $\div(f)$ is $K$-stable, and $f\in\CC(G)^{(B\times K)}$.

It follows that the extended map $\overline\rho$ is identified with the inclusion
\[
\Lambda(G/H) \cong \frac{(\CC(G)^{(B)})^H}{\CC^\times} \subseteq \frac{\CC(G)^{(B\times K)}}{\CC^\times}.
\]
The first two statements follow.

Consider now $(\gamma,(n_D)_{D\in\Delta})\in \Chi(C)\times\ZZ^\Delta$ as the class modulo $\CC^\times$ of a rational function $f\in \CC(G)^{(B\times K)}$, with $B$-eigenvalue $\lambda=\sigma(\gamma,(n_D)_{D\in\Delta})$ and $K$-eigenvalue $\omega=\tau(\gamma,(n_D)_{D\in\Delta})$. Then $\omega$ is trivial on $H$ if and only if $f\in\CC(G)^H$, if and only if $f$ is the pull-back on $G$ of a function in $\CC(G/H)$, if and only if $\lambda\in\Lambda(G/H)$.
\end{proof}

\begin{proposition}\label{prop:wonderfulclosure}
Let $H\subseteq G$ be a spherical subgroup, set $K=\overline H$ and $\Xi=\Span_\ZZ\Sigma(G/H)$. Then
\[
\Lambda(G/K) \subseteq \Xi \subseteq \Lambda(G/H).
\]
The subgroup $\widehat H\subseteq K$ associated with $\overline\rho(\Xi)$ via the map of Lemma~\ref{lemma:lunalattice} is the unique wonderful closure of $H$. It has the same dimension of $\overline H$, and it is the unique wonderful subgroup between $H$ and $\overline H$ that satisfies $\Sigma(G/H)=\Sigma(G/\widehat H)$. Moreover, the spherical closure of $\widehat H$ is $\overline H$.
\end{proposition}
\begin{proof}
The inclusion $\Xi\subseteq \Lambda(G/H)$ is obvious. The map $\pi^{H,K}_*\colon \N(G/H)\to \N(G/K)$ has kernel $\V^\ell(G/H)$, and satisfies $\pi_*(\V(G/H)) = \V(G/K)$. The other inclusion $\Lambda(G/K)\subseteq \Xi$ follows. Hence the subgroup $\widehat H$ contains $H$.

The lattice $\Lambda(G/\widehat H) = \Xi$ has basis $\Sigma(G/H)$ since the spherical roots of a spherical variety are linearly independent. Since $\Lambda(G/\widehat H)$ has finite index inside $\V^\ell(G/H)^\perp$ and $\pi^{H,\widehat H}(\V(G/H))=\V(G/\widehat H)$ we deduce that $\Sigma(G/H)=\Sigma(G/\widehat H)$.

If $\widetilde H$ is another wonderful subgroup such that $H\subseteq \widetilde H\subseteq \overline H$, then $\Lambda(G/\widetilde H)$ has also finite index in $\V^\ell(G/H)^\perp$, and $\Sigma(G/\widetilde H)$ is equal to $\Sigma(G/H)$ up to taking (positive) multiples of the elements of the latter. The dimension, minimality and uniqueness properties of $\widehat H$ follow, since $\Lambda(G/\widetilde H)\subseteq \Lambda(G/\widehat H)$ implies $\widetilde H\supseteq \widehat H$.

The last assertion follows from part (3) of Lemma~\ref{lemma:lunalattice}. Indeed, by definition we have $\tau^{-1}(\Chi(K)^{\widehat H})=\overline\rho(\Xi)$; now the composition $\sigma\circ\overline\rho$ is the identity on $\Lambda(G/H)$, which implies that $\sigma(\tau^{-1}(\Chi(K)^{\widehat H}))=\Xi$.

Suppose that $\Xi\cap \eS_G^\circ(G/K)\neq\varnothing$; then there exists a simple root $\alpha\in \Xi$ such that $2\alpha\in\Sigma(G/K)$. From $\ker\pi_*^{H,K}=\V^\ell(G/H)$ it also follows that $\alpha\in\Sigma(G/H)$. At this point, after \cite[Section 1.4]{Lu01}, there exist two colors of $G/H$ not stable under $P_\alpha$, and there exists only one such color in $G/K$, where $P_\alpha$ is the minimal parabolic subgroup of $G$ containing $B$ and associated with $\alpha$. This contradicts the assumption $\overline H=K$, therefore $\Xi\cap \eS_G^\circ(G/K)=\varnothing$ and the spherical closure of $\widehat H$ is $K$.
\end{proof}

\section{Automorphisms stabilizing all $G$-orbits}\label{s:allGorbits}
From now on, $X$ denotes a complete $G$-regular embedding, with open $G$-orbit $G/H$.

\begin{definition}
For any subset $\D\subseteq \partial_G X$ of $G$-stable prime divisors we define
\[
\Autz(X,\D) = \left\{ \phi\in\Autz(X)\,\middle\vert\, \phi(D)=D,\;\;\forall D\in\D \right\}.
\]
\end{definition}

Since $X$ is $G$-regular, the group $\Autz(X,\partial_G X)$ is also the connected group of automorphisms of $X$ stabilizing each $G$-orbit.

We recall now some results from \cite{BB96} (see also \cite{Br07}). The group $\Autz(X)$ is a linear algebraic group, with Lie algebra
\[
\Lie \Autz(X) = H^0(X,\T_X)
\]
where $\T_X$ is the sheaf of sections of the tangent bundle of $X$. The structure of $G$-module on $\Lie \Autz(X)$, induced by the adjoint action of $\theta_{G,X}(G)\subseteq \Autz(X)$, is given in \cite[Proposition 4.1.1]{BB96} in terms of global sections of the line bundles $\OO_X(D)$ where $D\in\partial_G X$.

Namely, there exists an exact sequence of $G$-modules
\begin{equation}\label{eqn:exact}
0 \to \Lie\Autz(X,\partial_G X) \to \Lie\Autz(X)\to\bigoplus_{D\in\partial_G X}\frac{H^0(X,\OO_X(D))}{\CC}\to 0.
\end{equation}
Moreover, for any $\D\in\partial_GX$ the Lie algebra of the subgroup $\Autz(X, \D)$ is the inverse image of the sum
\[
\bigoplus_{D\in(\partial_G X)\smallsetminus\D}\frac{H^0(X,\OO_X(D))}{\CC}.
\]

We recall that the exact sequence (\ref{eqn:exact}) is induced by the exact sequence of sheaves
\begin{equation}\label{eqn:exactsheaves}
0 \to \mathcal S_X \to \T_X \to \bigoplus_{D\in\partial_GX} \OO_X(D)\otimes_{\OO_X}\OO_{D} \to 0.
\end{equation}
Here $\mathcal S_X$ is the subsheaf of $\mathcal T_X$ preserving the ideal sheaf of $Y=\bigcup_{D\in\partial_GX}D$. We also have that $\mathcal S_X$ is generated by the image of the map $\mathcal O_X\otimes \Lie G\to \mathcal T_X$ induced by the $G$-action on $X$. The third map in the exact sequence (\ref{eqn:exactsheaves}) can be described locally in a neighborhood of a point $x\in X$ as follows:
\begin{equation}\label{eqn:vectorfields}
\sum_{i=1}^{m} a_i\partial_i \mapsto  \sum_i\left(\frac{a_i}{x_i}\otimes 1\right).
\end{equation}
Here we fix a regular system of parameters $x_1,\ldots,x_{\dim X}$ such that a local equation of $Y$ is $x_1\cdot\ldots\cdot x_m$. Moreover we denote by $(\partial_1,\ldots,\partial_{\dim X})$ the dual basis of the local basis $(dx_1,\ldots,dx_{\dim X})$ of $\Omega^1_X$ in $x$.

\begin{definition}
Let $0\neq\gamma\in\Chi(B)$. If it exists, we denote by $X(\gamma)$ the uniquely determined element of $\partial_G X$ such that $H^0(X,\OO_X(X(\gamma)))^{(B)}_\gamma\neq \varnothing$.
\end{definition}

A particular case of $\Autz(X)$ has been studied in \cite{Pe09}, where $X$ is a wonderful variety. Recall that $C$ acts trivially on any wonderful $G$-variety, hence we can consider $G'$-varieties without loss of generality. Moreover, in this case $\Autz(X,\D)$ is always semisimple and $X$ is wonderful under its action (see \cite[Theorem 2.4.2]{Br07}). It is possible to summarize the results of loc.cit.\ as follows.

\begin{theorem}\cite{Pe09}\label{thm:wonderful}
Let $\XX$ be a wonderful $G'$-variety and $\D\subseteq \partial_{G'} \XX$. Decompose $G'$ and $\XX$ into products
\[
G' = G'_1 \times \ldots \times G'_n, \quad \XX = \XX_1\times \ldots \times \XX_n,
\]
with a maximal number of factors in such a way that $G'_i$ acts non-trivially only on the factor $\XX_i$ for all $i\in\{1,\ldots,n\}$. Then
\[
\Autz(\XX,\D)= \Autz(\XX_1, \D_1) \times \ldots \times \Autz(\XX_n,\D_n),
\]
where $\D_i=\{ D \cap \XX_i \,|\, D\in \D\}\subseteq \partial_{G'_i} \XX_i$ and $\Autz(\XX_i,\D_i)$ acts only on $\XX_i$. Moreover, if the image of $G'_i$ in $\Autz(\XX_i,\D_i)$ is a proper subgroup, then $(G'_i,\XX_i)$ appears in the lists of ``exceptions'' of \cite[Sections 3.2 -- 3.6]{Pe09}. If $\D=\partial_{G'} \XX$ then all such exceptional factors have rank $0$ or $1$, and the rank of $\XX$ under the action of $G'$ and of $\Autz(\XX,\partial_{G'} \XX)$ coincide.
\end{theorem}

The invariants of $\XX_i$ under the action of $\Autz(\XX_i,\D_i)$ are also listed in loc.cit., for all cases where the image of $G'_i$ in $\Autz(\XX_i,\D_i)$ is a proper subgroup.

Now, the group $\Autz(X,\partial_G X)$ has been described in \cite{Br07}. Namely, thanks to \cite[Theorem 4.4.1, part (1)]{Br07}, there exists a split exact sequence of Lie algebras
\begin{equation}\label{eqn:partial}
0 \to \frac{\Lie \overline H}{\Lie H } \to \Lie\Autz(X,\partial_G X) \to \Lie\Autz(\XX(G/\overline H),\partial_G\XX(G/\overline H))\to 0.
\end{equation}
The same holds if we consider the variety $\XX(G/\widehat H)$ instead of $\XX(G/\overline H)$. Indeed, the results of \cite[Section 4.4]{Br07} hold (with same proofs) if we replace the spherical closure of $H$ with its wonderful closure.

It follows that $\Autz(X,\partial_G X)$ is reductive, its connected center is $(\overline H/H)^\circ=(\widehat H/H)^\circ$, and its semisimple part can be computed using Theorem \ref{thm:wonderful} and the lists of \cite{Pe09}. 

The invariants of $X$ under the action of $\Autz(X,\partial_G X)$ are easily recovered. Namely, if we choose a Borel subgroup of $\Autz(X,\partial_G X)$ containing the image of $B$, then colors (see \cite[Theorem 4.4.1, part (2)]{Br07}) and boundary divisors are the same of the $G$-action, which also induces a natural identification of $\Lambda_G(X)$ and $\Lambda_{\Autz(X,\partial_G X)}(X)$.

The spherical roots for the $\Autz(X,\partial_G X)$-action and the parabolic subgroup $\P_{\Autz(X,\partial_G X)}(X)$ are the same as those of $\XX(G/\widehat H)$; for the factors of $G$ that do not map surjectively onto factors of $\Autz(X,\partial_G X)$ these invariants are computed in \cite[Sections 3.2 -- 3.3]{Pe09}.

Finally, the structure of the orbits of $\Autz(X,\partial_G X)$ and of $G$ on $X$ coincide, which implies the equality of the fans of polyhedral convex cones of $X$ with respect to the two actions.

\begin{example}
Let $V$ be the irreducible $8$-dimensional $\Spin(7)$-module. Consider $G=\GG_m\times \Spin(7)$ acting on $X=\Bl_{[0,1]}(\PP(V\oplus\CC))$. Then $X$ is a regular $G$-variety, with one $G$-color (corresponding to an hyperplane in $V$) and two $G$-stable prime divisors (one given by $V$ itself, the other corresponding to the cone in $V$ over a smooth quadric in $\PP(V)$). We have $\Sigma_G(X)=\{\alpha_1+2\alpha_2+3\alpha_3\}$ and $\eS^p_G(X)=\{\alpha_1,\alpha_2\}$. The groups $\overline H$ and $\widehat H$ are equal, the variety $\XX=\XX(G/\overline H)$ is $\PP(V)=\PP^7$ and $G$ acts on it via its quotient $\PSO(7)$. It occurs in \cite[Section 3.3]{Pe09}; it has one $G$-color and only one $G$-stable prime divisor.

Notice that the natural map between the open $G$-orbits extends to a map $X\to \XX$, and that taking inverse images does not induce a bijection between $\partial \XX$ and $\partial X$. Nevertheless the automorphism groups stabilizing the whole boundary correspond, up to a torus factor and up to central isogeny: we have $\Autz(X,\partial X)=\GG_m\times \SO(8)$ and $\Autz(\XX,\partial \XX)=\PSO(8)$. The other invariants behave as above: the $\Autz(X,\partial X)$-colors and $\Autz(X,\partial X)$-stable prime divisors are the same of the $G$-action, and the same happens for $\XX$. Finally, we have $\Sigma_{\Autz(X,\partial X)}(X)=\Sigma_{\Autz(\XX,\partial \XX)}(\XX)=\{2\alpha_1+2\alpha_2+\alpha_3+\alpha_4\}$ and $\eS^p_{\Autz(X,\partial X)}(X)=\eS^p_{\Autz(\XX,\partial \XX)}(\XX)=\{\alpha_1,\alpha_2\}$.
\end{example}

\section{Relating $\Autz(X)$ to $\Autz(\XX)$}\label{s:relating}
From now on, $\XX=\XX(G/\widehat H)$ denotes the wonderful embedding of $G/\widehat H$. As a consequence of the last section, we assume that $\theta_{G,X}(G)=\Autz(X,\partial_G X)$ and that $\theta_{G,\XX}(G)=\Autz(\XX,\partial_G \XX)$. Indeed, if this is not the case we may first apply Theorem \ref{thm:wonderful} to $\XX$, replace $G_i'$ with the universal cover of $\Autz(\XX_i,\partial_{G_i} \XX_i)$ for all $i$ such that these groups are different, and then replace $C$ with $C\times (\widehat H/H)^\circ$.

\begin{lemma}\label{lemma:imageell}
The map $\pi^{H,\widehat H}\colon G/H\to G/\widehat H$ extends to a surjective $G$-equivariant map $\pi\colon X\to \XX$. An element $D\in\partial_GX$ is mapped surjectively onto $\XX$ if and only if $D\in(\partial_GX)^\ell$.
\end{lemma}
\begin{proof}
The fan $\F_G(\XX)$ contains the unique maximal cone $\V_G(\XX)$. Then, since $X$ is toroidal, the extended map $\pi\colon X\to \XX$ exists by \cite[Theorem 4.1]{Kn96}.

The condition $\pi(D)=\XX$ is equivalent to the condition that $\pi(Y)=G/\widehat{H}$, where $Y\subseteq X$ is the union of $G/H$ and the open $G$-orbit of $D$. This is also equivalent to the fact that the map $\pi^{H,\widehat H}$ extends to $Y\to G/\widehat H$.

The fan $\F_G(Y)$ has only one maximal cone, namely $\QQ_{\geq0}\rho_{G,X}(D)$, and the fan $\F_G(G/\widehat{H}$ has only the trivial cone $\{0\}$. Then, thanks to \cite[Theorem 4.1]{Kn96}, the map $\pi^{H,\widehat H}$ extends to $Y\to G/\widehat H$ if and only if $\rho_{G,X}(D)$ is in the kernel of $\ker\pi_*^{H,\widehat H}$, which is $\V^\ell(G/H)$.
\end{proof}

\begin{example}
Let $G=\SL(n+1)$ and $X=\Bl_p(\PP^{n+1})\times (\PP^n)^*$ as in Example~\ref{ex:blowup0}. We have $\XX=\PP^n\times(\PP^n)^*$, and the image $\pi(E)\subset \XX$ is the subset $\{([v],[\eta])\;|\; \eta(v)=0\}$, i.e.\ the unique $G$-stable prime divisor of $\XX$. The two other $G$-stable prime divisors of $X$ are both mapped surjectively onto $\XX$.
\end{example}

\begin{definition}
We denote by
\[
\xymatrix{
X \ar@{->}[r]^{\psi} & X' \ar@{->}[r]^{f} & \XX
}
\]
the Stein factorization of the map $\pi\colon X\to \XX$.
\end{definition}

In \cite[Section 4.4]{Br07} it is shown that $\Autz(X)$ acts on $X'$ in such a way that $\psi$ is equivariant; we denote the corresponding homomorphism as follows:
\[
\psi_*\colon \Autz(X)\to\Autz(X').
\]

Its kernel is the subgroup of automorphisms of $X$ stabilizing each fiber of $\psi$.

\begin{proposition}\label{prop:kerpsi}
The inclusions $\Z(\theta_{G,X}(G))^\circ\subseteq \ker\psi_* \cap \theta_{G,X}(G)\subseteq \Z(\theta_{G,X}(G))$ between subgroups of $\Autz(X)$ hold. Moreover, there is a local isomorphism
\begin{equation}\label{eqn:semidirprod}
\Autz\left(X,\left(\partial_G X\right)^{n\ell}\right) \cong \theta_{G,X}(G') \ltimes (\ker\psi_*)^\circ
\end{equation}
induced by the inclusion of both factors of the right hand side in $\Autz(X)$.
\end{proposition}
\begin{proof}
The first inclusion stems from the fact that $C=\Z(G)^\circ$ acts trivially on $\XX$, hence also on $X'$. On the other hand, if $g\in G$ stabilizes all fibers of $\psi$, then it acts trivially on $X'$ and also on $\XX$. Therefore, to show the second inclusion, we only have to check that no simple factor of $G$ acts trivially on $\XX$ but not on $X$. This is true because $\widehat H/H$ is abelian.

Let us prove the last statement. Both groups on the right hand side of (\ref{eqn:semidirprod}) are subgroups of $\Autz(X,(\partial_G X)^{n\ell})$: this is obvious for $\theta_{G,X}(G')$, so we only have to check it for $(\ker\psi_*)^\circ$. Notice that $\psi$ maps any element $D$ of $(\partial_G X)^{n\ell}$ onto a proper $G$-stable closed subset of $X'$ by Lemma~\ref{lemma:imageell}. It follows that $D$ is an irreducible component of $\psi^{-1}(\psi(D))$, hence it is stable under the action of $(\ker\psi_*)^\circ$. It also follows that $\psi_*$ maps $\Autz(X,(\partial_G X)^{n\ell})$ into $\Autz(X',\partial_G X')$.

The intersection $(\ker\psi_*)^\circ\cap\theta_{G,X}(G')$ is finite thanks to the first part of the proof, and $(\ker\psi_*)^\circ$ is a normal subgroup of $\Autz(X,(\partial_GX)^{n\ell})$. It only remains to prove that $\Autz(X,(\partial_GX)^{n\ell})$ is generated by $\theta_{G,X}(G')$ and $(\ker\psi_*)^\circ$.

By \cite[Theorem 4.4.1]{Br07}, we know that $\Autz(X',\partial X')$ and $\Autz(\XX,\partial \XX)$ are both semisimple and locally isomorphic. It follows that the universal cover of $\Autz(X',\partial X')$ acts on $\XX$ in such a way that $f$ is equivariant. On the other hand no element of this universal cover could act trivially on $X'$ and non-trivially on $\XX$, hence $\Autz(X',\partial X')$ itself acts on $\XX$, preserving all $G$-orbits. This produces a commutative diagram
\[
\xymatrix{
G \ar@{->}[r]^-{\theta_{G,X}} \ar@{->}[d]^-{\theta_{G,\XX}} & \Autz\left(X,\left(\partial_G X\right)^{n\ell}\right) \ar@{->}[d]^-{\psi_*} \\
\Autz(\XX,\partial_G \XX) \ar@{<-}[r]^-{f_*} & \Autz(X',\partial_G X')
}
\]
where $\theta_{G,\XX}$ is surjective by our assumptions. Therefore $\Autz(X,(\partial_GX)^{n\ell})$ is generated by $\theta_{G,X}(G)$ and $\ker(f_*\circ\psi_*)$. Notice that $f_*$ has finite kernel, that the kernel of $\psi_*$ contains $\theta_{G,X}(C)$, and that $\Autz(X,(\partial_GX)^{n\ell})$ is connected: we deduce that $\Autz(X,(\partial_GX)^{n\ell})$ is indeed generated by $\theta_{G,X}(G')$ and $(\ker\psi_*)^\circ$, and the proof is complete.
\end{proof}

If we denote by
\[
d\psi_*\colon\Lie\Autz(X)\to\Lie\Autz(X')
\]
the corresponding homomorphism of Lie algebras, then the following corollary is an immediate consequence of the above proposition.

\begin{corollary}\label{cor:dkerpsi}
The subspace $\ker d\psi_*\subseteq \Lie\Autz(X)$ is $G$-stable, and its intersection with $\Lie \theta_{G,X}(G)$ is equal to $\Lie \theta_{G,X}(C)$. There exists a $G$-equivariant splitting of the exact sequence (\ref{eqn:exact}) such that
\[
\ker d\psi_* = \Lie \theta_{G,X}(C)\oplus\bigoplus_{D\in(\partial_G X)^\ell}\frac{H^0(X,\OO_X(D))}{\CC}.
\]
\end{corollary}

\section{Restricting automorphisms of $X$ to fibers of $\psi$}\label{s:restricting}

We study now the automorphisms of a generic fiber of $\psi$ induced by automorphisms of $X$ belonging to $\ker\psi_*$. For this it is convenient to exploit the {\em local structure} of spherical varieties.

\begin{theorem}\cite[Theorem 2.3 and Proposition 2.4]{Kn94}
Let $Y$ be a spherical $G$-variety. Let $P_{G,Y}\supseteq B$ be the stabilizer in $G$ of the open $B$-orbit of $Y$, let $L_{G,Y}$ be the Levi subgroup of $P_{G,Y}$ containing $T$, and consider the following open subset of $Y$:
\[
Y_0 = Y\smallsetminus \bigcup_{D\in\Delta_G(Y)} D.
\]
Then there exists a closed $L_{G,Y}$-stable and $L_{G,Y}$-spherical subvariety $\Zi_{G,Y}$ of $Y_0$ such that the map
\[
\begin{array}{ccc}
P_{G,Y}^u \times \Zi_{G,Y} & \to &  Y_0 \\
(p,z) & \mapsto & pz
\end{array}
\]
is a $P_{G,Y}$-equivariant isomorphism, where $L_{G,Y}$ acts on $P_{G,Y}^u \times \Zi_{G,Y}$ by $l\cdot (p,z) = (lpl^{-1},lz)$. The commutator subgroup $(L_{G,Y},L_{G,Y})$ acts trivially on $\Zi_{G,Y}$, and if $Y$ is toroidal then every $G$-orbit meets $\Zi_{G,Y}$ in an $L_{G,Y}$-orbit.
\end{theorem}

\begin{definition}
We define $T_{G,Y}$ to be the quotient of $L_{G,Y}^r$ by the kernel of its action on $\Zi_{G,Y}$.
\end{definition}

We get back to our complete $G$-regular embedding $X$. The torus $T_{G,X}$ is a subquotient of $T$, and $\Zi_{G,X}$ is a spherical (toric) $T_{G,X}$-variety, with lattice $\Lambda_{T_{G,X}}(\Zi_{G,X})=\Chi(T_{G,X})=\Lambda_G(G/H)$ and fan of convex cones equal to $\F_G(X)$.

\begin{definition}
For any $x'$ in the open $G$-orbit of $X'$ we denote by $\kappa_{x'}$ the restriction map
\[
\kappa_{x'}\colon (\ker\psi_*)^\circ \to \Autz(X_{x'})
\]
where $X_{x'}=\psi^{-1}(x')$.
\end{definition}

Recall that $H$ is chosen in such a way that $BH$ is open in $G$, and $x_0=eH\in G/H\subseteq X$. 
Let us consider $x_0' = \psi(x_0)$: the fiber $X_{x_0'}$ is smooth and complete, and it is a toric variety under the action of the torus $S=(\widehat H/H)^\circ=H'/H$, where $H'$ is the stabilizer of $x_0'$.

Moreover, $S$ acts naturally on $G/H$ by $G$-equivariant automorphisms, and since $S$ is connected this $S$-action extends to $X$, stabilizing all colors of $X$ and all fibers of $\psi$. We may fix $\Zi_{G,X'}\subset X'$ containing $x_0'$, and choose $\Zi_{G,X}$ so that
\[
\Zi_{G,X}=\psi^{-1}(\Zi_{G,X'})\cap X_0,
\]
which implies that $\Zi_{G,X}$ contains $x_0$ and is stable under the action of $S$.

\begin{lemma}\label{lemma:Saction}
The above action of $S$ on $\Zi_{G,X}$ can be realized sending $S$ injectively into $T_{G,X}$, and then letting it act on $\Zi_{G,X}$ via the restriction of the usual action of $G$ on $X$.
\end{lemma}
\begin{proof}
If $nH\in S$ and $f\in\CC(G/H)^{(B)}_\chi$, then $gH\mapsto f(gnH)$ also belongs to $\CC(G/H)^{(B)}_\chi$, therefore there is a homomorphism (depending only on $\chi$) $\iota_\chi\colon S\to \GG_m$ such that $f(gnH) = \iota_\chi(n^{-1}H)f(gH)$ for all $g\in G$. This induces a homomorphism
\[
\iota\colon S\to\Hom(\Lambda_G(G/H),\GG_m) \cong T_{G,X},
\]
which can be shown to be injective, with image equal to the subtorus of $T_{G,X}$ corresponding to the subspace $\V_G^\ell(G/H)\subseteq \N_G(G/H)$ (see \cite[Proof of Theorem 4.3]{Br97}). Let us check that the restriction of the usual $T_{G,X}$-action on $\Zi_{G,X}$ to the subtorus $\iota(S)$ yields the action described above. The intersection $\Zi_{G,X}\cap G/H$ is dense in $\Zi_{G,X}$, and $\Zi_{G,X}$ is a toric $T_{G,X}$-variety with lattice equal to $\Lambda_G(G/H)$: it follows that $\iota(nH)gH=gnH$, because 
\[
f(\iota(nH)gH) = \chi(\iota(n^{-1}H)) f(gH) = \iota_\chi(n^{-1}H) f(gH) = f(gnH)
\]
for all $nH\in S$, $gH\in\Zi_{G,X}\cap G/H$, $\chi\in\Lambda_G(G/H)$ and $f\in\CC(G/H)^{(B)}_\chi$.
\end{proof}

The fiber $X_{x_0'}$ is also the fiber over $x_0'$ of the $S$-equivariant map $\Zi_{G,X}\to\Zi_{G,X'}$, which implies that its fan of convex cones is
\begin{equation}\label{eqn:fiberfan}
\F_{S}(X_{x_0'}) = \left\{c \;\middle\vert\; c\in \F_G(X),\; c\subset \V_G^\ell(G/H) \right\}.
\end{equation}

Since the $S$-boundary of $X_{x_0'}$ is given intersecting $X_{x_0'}$ with the elements of $(\partial_G X)^{\ell}$, there is an exact sequence of $S$-modules
\begin{equation} \label{eqn:seqfiber}
0 \to \Lie S \to \Lie\Autz(X_{x_0'})\to\bigoplus_{D\in(\partial_G X)^{\ell}}\frac{H^0(X_{x_0'},\OO_X(D\cap X_{x_0'}))}{\CC}\to 0.
\end{equation}

\begin{lemma}\label{lemma:automfiber}
Consider the homomorphism
\[
d\kappa_{x'}\colon\ker d\psi_*=\Lie (\ker\psi_*)^\circ\to \Lie\Autz(X_{x'})
\]
where $x'$ is in the open $B$-orbit of $X'$. If $V\subseteq\ker d\psi_*$ is a simple $G$-submodule, then $d\kappa_{x'}(V)=d\kappa_{x'}(\CC v)$, where $v\in V$ is a highest weight vector.
\end{lemma}
\begin{proof}
We may assume that $x'=x_0'$ and that $v = [s] \in H^0(X,\OO_X(D))/ \CC$ for some $D\in(\partial X)^\ell$, in view of Corollary~\ref{cor:dkerpsi}. From the expression in local coordinates (\ref{eqn:vectorfields}), it follows that $d\kappa_{x_0'}(v)$ is sent by the surjective map of (\ref{eqn:seqfiber}) to $[s|_{X_{x_0'}}]$, where $s|_{X_{x_0'}}\in H^0(X_{x_0'}, \OO_{X_{x_0'}}(D\cap X_{x'_0}))$ (see also the proof of \cite[Proposition 4.1.1]{BB96}).

If $s$ is a $B$-eigenvector then its zeros are $B$-stable. On the other hand, since $Bx_0'$ is open in $X'$, the only zeros of $s$ intersecting $X_{x_0'}$ are $G$-stable. It also follows that $(gs)|_{X_{x_0'}}$ and $s|_{X_{x_0'}}$ have the same zeros (hence are linearly dependent) for any $g\in G$ such that $gx_0$ doesn't lie on any color of $G/H$. This is true for $g$ lying in the dense subset $BH$ of $G$, and since $V$ is generated as a vector space by elements of the form $[gs]$ for $g\in BH$, the lemma follows.
\end{proof}

The following lemma on the structure of $\F_G(X)$ will be crucial in the study of the image of $\kappa_{x'}$.

\begin{lemma}\label{lemma:opposite}
Let $G$ be Let $i=1,2$ and $0\neq\gamma_i\in\Lambda_G(X)$ be such that $X(\gamma_i)$ exists, with $X(\gamma_i)\in(\partial X)^\ell$. Suppose that $\langle m,\gamma_1\rangle = -\langle m, \gamma_2\rangle$ for all $m\in\V^\ell_G(X)$. Then
\[
\langle \rho_{G,X}(D),\gamma_i \rangle = 0
\]
for all $i=1,2$, for all $D\in(\partial_G X)^{n\ell}$ and for all $D\in\Delta_G(X)$.
\end{lemma}
\begin{proof}
Consider the wonderful variety $\XX$. Both sets $\rho_{G,\XX}(\Delta_G(\XX))$ and $\rho_{G,\XX}(\partial_G\XX)$ generate $\N_G(\XX)$ as a vector space, and the convex cone generated by $\rho_{G,\XX}(\Delta_G(\XX))$ contains $-\rho_{G,\XX}(\partial_G\XX)$ (see \cite[Lemma 2.1.2]{Br07}). On the other hand, the set $\pi_*(\rho_{G,X}((\partial_G X)^{n\ell}))\subset\N_G(\XX)$ generates the same convex cone $\V_G(\XX)$ generated by $\rho_{G,\XX}(\partial_G\XX)$, and $\pi_*(\rho_{G,X}(\Delta_G(X)))=\rho_{G,\XX}(\Delta_G(\XX))$. It follows that there exists a linear combination
\[
v = \sum_{Y\in(\partial_G X)^{n\ell}} n_Y \rho_{G,X}(Y) + \sum_{Z\in\Delta_G(X)} n_Z \rho_{G,X}(Z) \in\V^\ell_G(X)
\]
where all the coefficients $n_Y$ and $n_Z$ are positive. From the assumptions on the characters $\gamma_i$, all the elements $\rho_{G,X}(Y)$ and $\rho_{G,X}(Z)$ above are non-negative on both $\gamma_1$ and $\gamma_2$: we deduce that $\langle v, \gamma_i\rangle \geq 0$, which yields  $\langle v, \gamma_i\rangle = 0$. The lemma follows.
\end{proof}





\section{Prime divisors not on the linear part of the valuation cone}\label{s:nonlinear}
In this section we study $\Autz(X,\D)$ under the assumption that $\D\supseteq (\partial_G X)^\ell$. We also suppose that $\D$ contains all the $G$-stable prime divisors $D$ that satisfy $H^0(X,\OO_X(D))=\CC$, since these prime divisors do not move under the action of $\Autz(X)$ anyway, and define $\E=\partial X\smallsetminus \D$.

Before stating the main result of this section, Theorem~\ref{thm:nonlinear}, we need to establish a correspondence between the divisors in $\E$ and certain boundary divisors of $\XX$.

Recall that, since $\XX$ is wonderful, the set $-\rho_{G,\XX}(\partial_X\XX)$ is a basis of $\N_G(\XX)$, dual to $\Sigma_G(\XX)$.

\begin{definition}
For an element $D\in\partial_G \XX$, we denote by $\sigma_D$ the spherical root of $\XX$ dual to $-\rho_{G,\XX}(D)$.
\end{definition}

Since $\Lambda_G(\XX)$ is a sublattice of $\Lambda_G(X)$, we consider $\sigma_D$ also as an element of the latter. Also recall that, thanks to \cite[Theorem 2.2.3]{Br07}, if $D\in\partial_G\XX$ satisfies $H^0(\XX,\OO_\XX(D))\neq\CC$ then $H^0(\XX,\OO_\XX(D))/\CC$ is irreducible with highest weight $\sigma_{D}$.

\begin{lemma}\label{lemma:movedareorthogonal}
Let $E\in\E$. Then:
\begin{enumerate}
\item\label{lemma:movedareorthogonal:piD} the image $\pi(E)$ is an element of $\partial_G\XX$, with $H^0(\XX,\OO_\XX(\pi(E)))\neq \CC$, and $E$ is the only element of $\partial_G X$ whose image is $\pi(E)$;
\item\label{lemma:movedareorthogonal:degree+orth} we have
\[
\pi_*(\rho_{G,X}(E)) = \rho_{G,\XX}(\pi(E)),
\]
and
\begin{equation}\label{eqn:orth}
\forall c \in \F_G(X)\smallsetminus \{c_{X,E}\}, c \mbox{ $1$-dimensional} \colon \quad c \subset \sigma_{\pi(E)}^\perp;
\end{equation}
\item\label{lemma:movedareorthogonal:simple} the $G$-modules $H^0(X,\OO_X(E))$ and $H^0(\XX,\OO_\XX(\pi(E)))$ are isomorphic.
\end{enumerate}
\end{lemma}
\begin{proof}
Let $\gamma\neq 0$ be such that $H^0(X,\OO_X(E))^{(B)}_\gamma\neq\varnothing$. Since $E\in(\partial_GX)^{n\ell}$, the character $\gamma$ is non-negative on $\rho_{G,X}((\partial_GX)^\ell)$, which generates the whole $\V^\ell_G(X)$ as a convex cone, because $X$ is complete. It follows that $\gamma\in\V^\ell_G(X)^\perp$, which implies that some positive integral multiple of $\gamma$, say $n\gamma$, lies in  $\Lambda_G(\XX)$. Let us also assume that it is indecomposable in $\Lambda_G(\XX)$, i.e.\ that $n$ is minimal satisfying $n>0$ and $n\gamma\in\Lambda_G(X)$.

Consider $\pi(E)$: it is a proper subset of $\XX$ because $\pi_*(\rho_{G,X}(E))\neq 0$. If it is not a $G$-stable prime divisor of $\XX$, then $\pi_*(\rho_{G,X}(E))$ doesn't lie on any $1$-dimensional face of $\V_G(\XX)$. On the other hand, each element of $\partial_G \XX$ is the image $\pi(D)$ of some $G$-stable prime divisor $D$ of $X$, with $\pi_*(\rho_{G,X}(D))$ equal to a positive rational multiple of $\rho_{G,\XX}(\pi(D))$. This implies that $n\gamma\in\Lambda_G(\XX)$ is non-negative on $\rho_{G,\XX}(\partial_G\XX)$ and negative on $\pi_*(\rho_{G,X}(E))$, which is absurd because $\rho_{G,\XX}(\partial_G\XX)$ generates $\V_G(\XX)$ as a convex cone.

We conclude that $\pi(E)\in\partial_G\XX$, and that $E$ is the unique element of $\partial _GX$ whose image is $\pi(E)$, because $n\gamma$ is non-negative on $\rho_{G,X}(E')$ for any $E'\in\partial_GX$ different from $E$.  Let $0 > -m = \langle \rho_{G,\XX}(\pi(E)),n\gamma\rangle$. Then $H^0(\XX,\OO_\XX(m\pi(E)))\neq \CC$.

From \cite[Theorem 2.2.3]{Br07} it follows that $H^0(\XX,\OO_\XX(\pi(E)))\neq\CC$, that the quotient $H^0(\XX,\OO_\XX(\pi(E)))/\CC$ is irreducible with highest weight $\sigma_{\pi(E)}$, and that any $\chi\in\Lambda_G(\XX)$ satisfying
\begin{equation}\label{eqn:}
\langle \rho_{G,\XX}(D), \chi \rangle \geq 0 \quad \forall D\in(\partial_G \XX\smallsetminus\{\pi(E)\}\cup \Delta_G(\XX), \qquad \langle \rho_{G,\XX}(\pi(E)), \chi \rangle < 0
\end{equation}
is a positive multiple of $\sigma_{\pi(E)}$. We have then shown (\ref{lemma:movedareorthogonal:piD}). It also follows that $n\gamma$ is a positive multiple of $\sigma_{\pi(E)}$, whence $\gamma$ is non-positive on $\V_G(X)$ and so it is zero on $\rho_{G,X}(D)$ for all $D\in\partial_GX$ different from $E$. This shows (\ref{eqn:orth}).

Now recall that $n\gamma$ is indecomposable in $\Lambda_G(\XX)$. Since it is a positive multiple of $\sigma_{\pi(E)}$, it is equal to $\sigma_{\pi(E)}$. On the other hand $\Sigma_G(X)=\Sigma_G(\XX)$ and $\sigma_{\pi(E)}$ is also indecomposable in $\Lambda_G(X)$. Therefore $n=1$, and we have
\[
\langle \rho_{G,X}(E), \gamma \rangle = \langle \pi_*(\rho_{G,X}(E)), \gamma \rangle = -1 = \langle \rho_{G,\XX}(\pi(E)),\sigma_{\pi(E)}\rangle,
\]
whence $\pi_*(\rho_{G,X}(E)) = \rho_{G,\XX}(\pi(E))$. The proof of part (\ref{lemma:movedareorthogonal:degree+orth}) is complete.

Finally, since $\gamma$ is the highest weight of an arbitrary non-trivial $G$-submodule of $H^0(X,\OO_X(E))$ and the latter is multiplicity-free since $X$ is spherical, the proof of (\ref{lemma:movedareorthogonal:simple}) is also complete.
\end{proof}

Notice that property (\ref{lemma:movedareorthogonal:degree+orth}) is a rather strong condition on the fan $\F_G(X)$. It implies that any maximal cone $c$ of $\F_G(X)$ is generated by $\rho_{G,X}(E)$ together with the intersection $c\cap\sigma_{\pi(E)}^\perp$.

\begin{example}\label{ex:blowup1}
Let us illustrate Lemma~\ref{lemma:movedareorthogonal} in an example. Consider again $G=\SL(n+1)$ acting on $X=\Bl_p(\PP^{n+1})\times (\PP^n)^*$ as in Example~\ref{ex:blowup0}. Set $\D=(\partial X)^\ell=\{D_1,D_2\}$, and recall that $X$ has a third $G$-invariant prime divisor $E$. The prime divisor $\pi(E)=\{([v],[\eta])\;|\; \eta(v)=0\}$ of $\XX=\PP^n\times(\PP^n)^*$ is not stable under the action of $\Autz(\XX) = \PSL(n+1)\times\PSL(n+1)$. The spherical root $\sigma_{\pi(E)}$ is the unique element $\sigma$ of $\Sigma_G(X)$. As in Lemma~\ref{lemma:movedareorthogonal}, all maximal cones of $\F_G(X)$ are generated by $\rho(E)$ together with their intersection with $\sigma^\perp$, which is here the linear part of $\V_G(X)$.
\end{example}

\begin{definition}
We denote by
\[
\Lambda_G(X,\E) \subseteq \Lambda_G(X)
\]
the sublattice generated by the elements $\sigma_{\pi(E)}$ for all $E\in\E$.
\end{definition}

\begin{corollary}\label{cor:sum}
\[
\Lambda_G(X) = \rho_{G,X}(\E)^\perp \oplus \Lambda_G(X,\E).
\]
\end{corollary}
\begin{proof}
From Lemma \ref{lemma:movedareorthogonal} we deduce that for all $E\in\E$ the element $\rho_{G,X}(E)$ is $-1$ on the spherical root $\sigma_{\pi(E)}$ of $X$, and zero on all other spherical roots of $X$. The corollary follows.
\end{proof}

\begin{remark}
In the proof of Lemma~\ref{lemma:movedareorthogonal} we used the crucial fact that $X$ and $\XX(G/\widehat H)$ have the same spherical roots. The decomposition of $\Lambda(G/H)$ into the above direct sum would indeed be false in general, if we had used $\XX(G/\overline H)$ instead of $\XX(G/\widehat H)$.
\end{remark}

\begin{definition}
Define
\[
\EE = \left\{\pi(E) \;\middle\vert\;E\in\E \right\},
\]
and
\[
\DD = \partial\XX\smallsetminus\EE.
\]
\end{definition}

\begin{definition}\label{def:Anl}
Let $A'=A'(X,\D)$ be the universal cover of the semisimple group $\Autz(\XX,\DD)$, and $A=A(X,\D)=A'\times C$. We denote by
\[
\vartheta'\colon G' \to A'
\]
the lift of $\theta_{G',\XX}\colon G'\to\Autz(\XX,\DD)$ to $A'$, and we set
\[
\vartheta = \vartheta'\times \id_C \colon G\to A.
\]
We also choose a Borel subgroup $B_A$ of $A$ such that $B_A\supseteq \vartheta(B)$.
\end{definition}

Now we are ready to state the main result of this section.

\begin{theorem}\label{thm:nonlinear}
Under the assumptions listed at the beginning of this section, and with the above notations:
\begin{enumerate}
\item\label{thm:nonlinear:lift} The action of $A(X,\D)$ lifts from $\XX$ to $X$, and the image of $A(X,\D)$ inside $\Autz(X)$ is equal to $\Autz(X,\D)$.
\item As an $A=A(X,\D)$-variety, $X$ is $G$-regular with boundary $\D$.
\item\label{thm:nonlinear:restriction} Via the inclusion $B_A\supseteq \vartheta(B)$ the lattice $\Lambda_A(X)$ is identified with the lattice $\rho_{G,X}(\E)^\perp\subseteq \Lambda_G(X)$, and $\N_A(X)$ with $\Lambda_G(X,\E)^\perp\subseteq \N_G(X)$.
\item\label{thm:nonlinear:colors} The set $\Delta_A(X)$ is in natural bijection with $\Delta_G(X)$, in such a way that $\rho_{A,X}(D)=\rho_{G,X}(D)|_{\rho_{G,X}(\E)}$.
\item\label{thm:nonlinear:PeSigma} The parabolic subgroup $\P_A(X)$ is equal to $\P_A(\XX)$, and we have $\Sigma_A(X)=\Sigma_A(\XX)$. Both can be computed using Theorem~\ref{thm:wonderful} and \cite[Sections 3.2 -- 3.6]{Pe09}.
\item Each cone of $\F_G(X)$ intersects $\Lambda_G(X,\E)^\perp$ in a face, and the set of these intersections for all cones of $\F_G(X)$ is the  fan $\F_A(X)$.
\end{enumerate}
\end{theorem}

The proof occupies the rest of the section. In particular, part (\ref{thm:nonlinear:PeSigma}) follows immediately from part (\ref{thm:nonlinear:lift}). The first statement of part (\ref{thm:nonlinear:colors}) follows from the fact that the $G$-colors and the $A$-colors of $\XX$ coincide (see \cite[Theorem 2.4.2 (2)]{Br07}), and their inverse images under $\pi$ are the colors of $X$. The second statement of part (\ref{thm:nonlinear:colors}) follows from the first, and from part (\ref{thm:nonlinear:restriction}). The rest follows from Lemma~\ref{lemma:restrictions2}, Theorem~\ref{thm:nonlinear:lifts}, and Corollary~\ref{cor:nonlinear:image}.

\begin{example}\label{ex:blowup2}
Let us discuss the above theorem in the setting of Example~\ref{ex:blowup1}. Here $\DD=\varnothing$, and notice that the action of $\Autz(\XX,\DD)=\PSL(n+1)\times\PSL(n+1)$ does not lift to an action on $X=\Bl_p(\PP^{n+1})\times (\PP^n)^*$, whereas the action of the universal cover $A'=\SL(n+1)\times \SL(n+1)$ does. Moreover $C$ is trivial, so $A=A'$.

Under the action of $A$ the variety $X$ is toroidal of rank $1$; its lattice is generated by the last fundamental weight of the first factor $\SL(n+1)$ of $A$, and this lattice is identified as in the theorem with the one-dimensional sublattice $\rho_{G,X}(\E)^\perp = \ZZ\omega^G_n$ of $\Lambda_G(X)$. The other summand $\Lambda_G(X,\E)$ of $\Lambda_G(X)$ is the lattice $\ZZ\sigma=\ZZ(\omega_1+\omega_n)$.

As a consequence $\N_A(X)$ is identified with the one-dimensional subspace $\sigma^\perp$ of $\N_G(X)$, the fan $\F_A(X)$ is the unique fan with two (opposite) cones of dimension $1$, and is obtained from $\F_G(X)$ by intersecting all cones with $\sigma^\perp$.
\end{example}

In view of proving part (\ref{thm:nonlinear:lift}) of Theorem~\ref{thm:nonlinear}, we start finding a candidate for a generic stabilizer of the $A$-action on $X$. Let $\widehat H_A\subseteq A$ be%
\footnote{Our notation is consistent thanks to Corollary \ref{cor:H_A}.}
the stabilizer of the point $e\widehat H\in G/\widehat H \subseteq \XX$. Then we have $\vartheta(\widehat H) = \widehat H_A \cap \vartheta(G)$.

We also notice that thanks to our general assumptions any $G$-linearization of an invertible sheaf $\XX$ can be uniquely extended to an $A$-linearization, inducing an identification of the two groups $\Pic^G(\XX)$ and $\Pic^A(\XX)$.

\begin{lemma}\label{lemma:restrictions}
\begin{enumerate}
\item\label{lemma:restrictions:injective} The pull-back of characters of $B_A$ along $\vartheta|_{B}$ induces an injective map $r\colon \Lambda_A(\XX) \to \Lambda_G(\XX)$. It maps $\Sigma_A(\XX)$ to the set of spherical roots $\{\sigma_D \,|\, D\in\DD\}$.
\item \label{lemma:restrictions:V} The dual map $r^*\colon\N_G(\XX)\to \N_A(\XX)$ satisfies
\[
r^*(V_G(\XX)) = V_A(\XX).
\]
\item \label{lemma:restrictions:openorbit} We have that $\partial_A \XX=\D$, and
\[
A/\widehat H_A = \XX \smallsetminus \bigcup_{D\in \DD} D.
\]
\item \label{lemma:restrictions:surjective} The pull-back of characters of $\widehat H_A$ along $\vartheta|_{\widehat H}$ is a surjective homomorphism $r'\colon\Chi(\widehat H_A)\to\Chi(\widehat H)$ with free kernel of rank $|\EE|$.
\end{enumerate}
\end{lemma}
\begin{proof}
The injectivity of the map $r$ is obvious, since it corresponds to taking a $B_A$-semiinvariant $f\in\CC(\XX)$ and considering it as a $B$-semiinvariant. The rest of part (\ref{lemma:restrictions:injective}) follows from the results of \cite{Pe09}, and it can also be shown directly using the following fact: the spherical roots of $\XX$ are the $T$-weights appearing in the quotient of tangent spaces
\[
\frac{\mathrm T_z\XX}{\mathrm T_z(G\,z)},
\]
where $z\in\XX$ is the unique fixed point of $B^-$. Let us choose a maximal torus $T_A$ of $A$ containing $\vartheta(T)$: if $B^-_A\subseteq A$ is the Borel subgroup satisfying $B_A\cap B^-_A = T_A$ then $B^-_A$ contains $\vartheta(B^-)$. Hence $z$ is also the unique $B^-_A$-fixed point, therefore the spherical roots of $\XX$ as an $A$-variety are the $T_A$-weights appearing in the quotient of tangent spaces
\[
\frac{\mathrm T_z\XX}{\mathrm T_z(A\,z)},
\]
form the set $\Sigma_A(\XX)$. This implies part (\ref{lemma:restrictions:injective}), and part (\ref{lemma:restrictions:V}) is an immediate consequence.

The first statement of part (\ref{lemma:restrictions:openorbit}) stems from the fact that each $E\in\EE$ is not stable under the action of $A$, and the second follows from the first because $\XX$ is wonderful under the action of $A$.

For part (\ref{lemma:restrictions:surjective}), we notice that $r'$ can be identified with the natural map
\[
\frac{\Chi(C)\times \ZZ^\Delta}{\overline\rho_{A,\XX}(\Lambda_A(\XX))} \to \frac{\Chi(C)\times \ZZ^\Delta}{\overline\rho_{G,\XX}(\Lambda_G(\XX))}
\]
(see diagram (\ref{eqn:diagram})). The kernel of $r'$ is then $\Lambda_G(\XX)/r(\Lambda_A(\XX))$ which is free, generated by the spherical roots $\sigma_E$ for all $E\in\EE$ by part (\ref{lemma:restrictions:injective}).
\end{proof}

Let us put together two copies of the diagram (\ref{eqn:diagram}), one for the $G$- and one for the $A$-action, also adding the extensions of $\overline\rho_G$ and $\overline\rho_A$ resp.\ to $\Lambda_G(G/H)$ and $\Lambda_A(A/\widehat H_A)$, as in Section~\ref{s:wclosure}. We obtain a commutative diagram
\[
\xymatrixrowsep{8pt}
\xymatrixcolsep{13pt}
\xymatrix{
\Lambda_A(A/K_A) \ar@{^{(}->}[r] &
\Lambda_A(A/\widehat H_A) \ar@{^{(}->}[rr]^-{\overline\rho_A} \ar@{^{(}->}[ddd]^-{r}&
 &
\Chi(C)\times\ZZ^\Delta \ar@{->>}[r]^-{\tau_A} \ar[dr]_-{\sigma_A} \ar@{=}[ddd] &
\Chi(K_A) \ar@{->>}[r] &
\Chi(\widehat H_A) \ar@{->>}[ddd]^-{r'}\\
 & &  &  & \Chi(B_A) \ar[d]^-{(\vartheta|_B)^*} &  \\
 & & &  & \Chi(B) &  \\
\Lambda_G(G/K) \ar@{^{(}->}[r] &
\Lambda_G(G/\widehat H) \ar@{^{(}->}[r] &
\Lambda_G(G/H) \ar@{^{(}->}[r]^-{\overline\rho_G} &
\Chi(C)\times\ZZ^\Delta \ar@{->>}[r]^-{\tau_G} \ar[ur]^-{\sigma_G} &
\Chi(K) \ar@{->>}[r] &
\Chi(\widehat H) \\
}
\]
where $K_A$ is the spherical closure of $\widehat H_A$, and $K$ is the spherical closure of $\widehat H$ (and of $H$).
The last arrow of the first row is the restriction map, which can be seen as the quotient
\[
\Chi(K_A) \to \Chi(K_A)/\Chi(K_A)^{\widehat H_A}\cong \Chi(\widehat H_A).
\]
The same remark holds for the last map of the second row and the groups $K, \widehat H$.

In order to determine a generic stabilizer in $A$ for $X$, we start defining a lattice $\Lambda\subseteq \Lambda_G(G/H)$. A posteriori, it is the lattice of $B$-eigenvalues of $B_A$-eigenvectors $f\in\CC(X)^{(B_A)}$. Such a function $f$ cannot have zeros nor poles on the divisors in $\E$, since these are not $A$-stable, nor are $A$-colors of $X$. This suggests the following definition of $\Lambda$.

\begin{definition}
Let $\Lambda$ be the lattice
\[
\Lambda = \rho_{G,X}(\E)^\perp \, \subseteq \Lambda_G(G/H).
\]
\end{definition}

\begin{proposition}\label{proposition:H_A}
The following inclusion holds:
\[
\overline\rho_G(\Lambda)\supseteq \overline\rho_A(\Lambda_A(A/K_A)).
\]
The subgroup $H_A$ of $K_A$ corresponding to the lattice $\overline\rho_G(\Lambda)$ is a spherical subgroup of $A$. We have $\overline{H_A}=K_A$ and $\vartheta(H) = H_A \cap \vartheta(G)$. This induces a $G$-equivariant identification of $G/H$ with an open subset of $A/H_A$.
\end{proposition}
\begin{proof}
Let $\chi\in\Lambda_A(A/K_A)\subseteq \Lambda_A(A/\widehat H_A)$. If $f \in\CC(A/\widehat H_A)^{(B_A)}_\chi$, then consider its pull-back on $X$, denoted by $\widetilde f$. It is also a $B$-eigenvector with $B$-eigenvalue $\widetilde\chi= r(\chi)$.

We know that the divisor $\div(\widetilde f)$ on $X$ is $B_A$-stable, so in general it is a linear combination of colors and $A$-stable prime divisors. In any case, its components do not belong to $\E$, because the latter consists of prime divisors moved by $A$. It follows that all discrete valuations in $\V_G(G/H)$ coming from these elements of $\E$ must take the value $0$ on $\widetilde\chi$.

Therefore $\widetilde\chi\in\Lambda$, and the first assertion is proved. In order to verify that $H_A$ is spherical we have to show that $\sigma_A$ restricted to $\overline\rho_G(\Lambda)=\tau_A^{-1}\left(\Chi(K_A)^{H_A}\right)$ is injective. But we already know that the restriction of $\sigma_G$ on $\overline\rho_G(\Lambda_G(G/H))$ is injective, and that $\Lambda\subseteq \Lambda_G(G/H)$: this proves the second assertion.

The equality $\overline{H_A}=K_A$ stems from part (3) of Lemma~\ref{lemma:lunalattice}, as soon as we check its hypothesis regarding the set $\eS_A^\circ(A/K_A)$, i.e.\ we have to check that $\sigma_A(\overline\rho_G(\Lambda))\cap \eS_A^\circ(A/K_A)=\varnothing$. 

Let $\alpha\in \eS_A^\circ(A/K_A)$. We claim that $r(\alpha)\in \eS_G^\circ(G/K)$, i.e.\ the image $r(2\alpha)$ is a spherical root $2\alpha'$ of $G/K$ for some simple root $\alpha'$ of $G$, such that $\langle \alpha', \beta^\vee\rangle$ is even for all simple root $\beta$ of $G$ such that $2\beta\in\Sigma_G(G/K)$. This stems from an elementary case-by-case check on the classification of \cite{Pe09}: first of all if $2\alpha$ is a spherical root of $\XX$ for the action of $\Autz(\XX,\DD)$ there is an indecomposable factor $\XX'$ of $\XX$ (see \cite[Definition 2.2.1]{Pe09}) with $2\alpha$ among its spherical roots. Then one checks that either the factors of $\theta_{G,\XX}(G)$ and of $\Autz(\XX,\DD)$ acting non-trivially on $\XX'$ are isomorphic, or $\XX'$ is equal to the case $\mathbf 2_{rk=2}$ of \cite{Pe09}.

In the first case we have immediately $r(\alpha)\in \eS_G^\circ(G/K)$. In the second case we have $r(2\alpha)=2\alpha'\in\Sigma_G(\XX)$ for a simple root $\alpha'$ of $G$. Moreover, for {\em all} $\beta\in\Sigma_G(\XX)$ the value $\langle \alpha', \beta^\vee\rangle$ is even. This implies again $r(\alpha)\in \eS_G^\circ(G/K)$.

By Lemma~\ref{lemma:lunalattice} applied to the inclusion $H\subseteq K$ we have $\sigma_G(\overline\rho_G(\Lambda_G(G/H)))\cap \eS_G^\circ(G/K)=\varnothing$. Therefore we have $\sigma_A(\overline\rho_G(\Lambda))\cap \eS_A^\circ(A/K_A)=\varnothing$, and hence $\overline{H_A}=K_A$.

Next, we claim that $r'$ induces an isomorphism between $\Chi(\widehat H_A)^{H_A}$ and $\Chi(\widehat H)^{H}$. This shows that $\widehat H_A/H_A\cong \widehat H/H$, and the rest of the lemma follows. To prove the claim, it is enough to notice that
\begin{eqnarray*}
\Chi(\widehat H_A)^{H_A} & \cong & \frac{\overline\rho_G(\Lambda)}{\ker\tau_A}\\
& = & \frac{\overline\rho_G(\Lambda)}{\overline\rho_G(r(\Lambda_A(A/\widehat H_A)))}\\
& \cong & \frac{\overline\rho_G(\Lambda)\oplus\overline\rho_G(\Lambda_G(X,\E))}{\overline\rho_G(\Lambda_G(G/\widehat H))}\\
& = & \frac{\overline\rho_G(\Lambda\oplus\Lambda_G(X,\E))}{\overline\rho_G(\Lambda_G(G/\widehat H))}\\
& = & \frac{\overline\rho_G(\Lambda_G(G/H))}{\overline\rho_G(\Lambda_G(G/\widehat H))}\\
& \cong & \Chi(\widehat H)^H,
\end{eqnarray*}
and that the resulting isomorphism $\Chi(\widehat H_A)^{H_A}\cong \Chi(\widehat H)^{H}$ is indeed induced by $r'$.
\end{proof}

We build an embedding $X_A$ of $A/H_A$, and then prove that we actually obtain $X$. 

\begin{lemma}\label{lemma:restrictions2}
The pull-back of characters of $B_A$ to $B$ along $\vartheta|_B$ induces an injective map $s\colon\Lambda_A(A/H_A)\to\Lambda_G(G/H)$ whose image is $\Lambda$. The dual map $s^*\colon \N_G(G/H) \to \N_A(A/H_A)$ satisfies
\[
s^*(\V_G(G/H))= \V_A(A/H_A),
\] 
and induces an isomorphism
\[
s^*|_{\V^\ell_G(G/H)}\colon\V^\ell_G(G/H) \to \V^\ell_A(A/H_A).
\]
\end{lemma}
\begin{proof}
Let $\gamma\in\Lambda_A(A/H_A)$: it is the $B_A$-eigenvalue of a $B_A$-eigenvector $f\in\CC(A/H_A)^{(B_A)}$. But $f$ is a $B$-eigenvector too and the character $\chi=s(\gamma)$ is its $B$-eigenvalue. Both the $B$- and the $B_A$-eigenvalue determine $f$ up to a multiplicative constant, hence $s$ is injective.

Consider the commutative diagram
\[
\xymatrix{
\Lambda_A(A/H_A) \ar@{^{(}->}[r]^-{\overline\rho_A} \ar@{^{(}->}[d]^{s} &
\Pic^A(\XX) \ar@{=}[d] \\
\Lambda_G(G/H) \ar@{^{(}->}[r]^-{\overline\rho_G} &
\Pic^G(\XX)
}
\]
From the definition of $H_A$ we have $\overline\rho_A(\Lambda_A(A/H_A))=\overline\rho_G(\Lambda)$, therefore we obtain $s(\Lambda_A(A/H_A))=\Lambda$.

Let $v\in\V_A(A/H_A)$. It corresponds to an $A$-invariant valuation, which is {\em a fortiori} $G$-invariant too: in other words we can compute $v$ also on $\Lambda_G(G/H)$ obtaining an element of $\V_G(G/H)$. This shows that $s^*(\V_G(G/H))\supseteq \V_A(A/H_A)$.

Then we notice that $s$ extends the map $r$ of Lemma~\ref{lemma:restrictions}. This gives the commutative diagram
\[
\xymatrixcolsep{40pt}
\xymatrix{
\N_G(G/H) \ar@{->>}[r]^{\pi^{H,\widehat H}_*} \ar@{->>}[d]^{s^*} &
\N_G(G/\widehat H) \ar@{->>}[d]^{r^*} \\
\N_A(A/H_A) \ar@{->>}[r]^{\pi^{H_A,\widehat H_A}_*}  &
\N_A(A/\widehat H_A)
}
\]
where $V_A(A/H_A)$ (resp.\ $V_G(G/H)$) is the inverse image of $V_A(A/\widehat H_A)$ (resp.\ of $V_G(G/\widehat H)$) thanks to Lemma \ref{lemma:inclusion}.

This, together with Lemma \ref{lemma:restrictions}, part (\ref{lemma:restrictions:V}), proves $s^*(\V_G(G/H))= \V_A(A/H_A)$. The image of $\V^\ell_G(G/H)$ is contained in $\V^\ell_A(A/H_A)$, and we conclude the proof observing that the dimensions of $\V^\ell_G(G/H)$ and $\V^\ell_G(A/H_A)$ are both equal to the dimension of $\widehat H/H\cong \widehat H_A/H_A$.
\end{proof}
\begin{corollary}\label{cor:H_A}
The wonderful closure of $H_A$ is $\widehat H_A$.
\end{corollary}
\begin{proof}
By construction $H_A\subseteq \widehat H_A \subseteq K_A=\overline{H_A}$. From Lemma~\ref{lemma:restrictions2} we deduce that $A/H_A$ and $A/\widehat H_A$ have the same spherical roots: the corollary follows then from Proposition~\ref{prop:wonderfulclosure}.
\end{proof}

We shall now define the fan of convex cones of $X_A$, using that of $X$. First, we collect some consequences on $\F(X)$ of the analysis we have developed so far.
\begin{definition}
Let $\F$ be a fan of convex cones, consider a subset $\F'\subset \F$ and let $c\in\F\smallsetminus\F'$ be $1$-dimensional. Then $\F$ is the {\em join} of $\F'$ and $c$ if each element of $\F\smallsetminus\F'$ is the convex cone generated by $c$ and an element of $\F'$.
\end{definition}
\begin{corollary}\label{cor:fan}
\begin{enumerate}
\item\label{cor:fan:sigma} Let $E\in\E$, and let $\F_G^{\sigma_{\pi(E)}}(X)$ be the fan of convex cones obtained intersecting each element of $\F_G(X)$ with $\sigma_{\pi(E)}^\perp$. Then $\F_G(X)$ is the join of $\F_G^{\sigma_{\pi(E)}}(X)$ and $c_{X,E}$.
\item\label{cor:fan:injective} Let $\F_G^\Lambda(X)$ be the fan of convex cones obtained intersecting each element of $\F_G(X)$ with $\Lambda_G(X,\E)^\perp$. Then the restriction of $s^*$ to $\supp\F_G^\Lambda(X)$ is injective, and $s^*(\supp\F_G^\Lambda(X)) = \V_A(A/H_A)$.
\item\label{cor:fan:smoothcomplete} The set
\[
\left\{ s^*(c) \,\middle\vert\, c\in\F_G^\Lambda(X) \right\}
\]
is a fan of polyhedral convex cones in $\N_A(A/H_A)$. The associated embedding of $A/H_A$ is smooth and complete.
\end{enumerate}
\end{corollary}
\begin{proof}
Part (\ref{cor:fan:sigma}) follows from Lemma \ref{lemma:movedareorthogonal}, part (\ref{lemma:movedareorthogonal:degree+orth}). Part (\ref{cor:fan:injective}) follows from part (\ref{cor:fan:sigma}) applied to all $E\in\E$, together with Corollary \ref{cor:sum} and Lemma \ref{lemma:restrictions2}. We turn to part (\ref{cor:fan:smoothcomplete}). Completeness of this embedding is an immediate consequence of part (\ref{cor:fan:injective}). For smoothness, we observe that a maximal cone $c$ of $\F_G(X)$ can be written as
\[
c = (\left\{ -\sigma_E \,\middle\vert\, E\in\E \right\}\cup \Psi)^{\geq0}
\]
where $\Psi$ is a basis of $\Lambda=\rho_{G,X}(\E)^\perp$, thanks to the smoothness of $X$ together with part (\ref{cor:fan:sigma}) applied to all $E\in\E$ and Corollary \ref{cor:sum}. Therefore
\[
s^*\left(c\cap\left(\Lambda_G(X,\E)^\perp\right)\right) = \left(s^{-1}\left(\Psi\right)\right)^{\geq0}.
\]
The smoothness characterization recalled in Section~\ref{s:definitions} is verified, since $s^{-1}\left(\Psi\right)$ is a basis of $\Lambda_A(A/H_A)$, and the proof is complete. 
\end{proof}

\begin{definition}
We define
\[
\F_A = \left\{ s^*(c) \,\middle\vert\, c\in\F_G^\Lambda(X) \right\},
\]
and we denote by $X_A$ the corresponding embedding of $A/H_A$.
\end{definition}

\begin{theorem}\label{thm:nonlinear:lifts}
The inclusion $G/H\subseteq A/H_A$ extends to an $A$-equivariant isomorphism between $X$ and $X_A$.
\end{theorem}
\begin{proof}
The group $G$ acts on $X_A$ via the map $\theta$, and it is enough to show $X_A$ is a toroidal embedding of $G/H$ with fan $\F_G(X)$. Let us first prove this fact with the assumption that $|\E|=1$, say $\E= \{E\}$. 

In addition to the $G$-equivariant map $\pi\colon X\to \XX$ we also have by construction an $A$-equivariant map $\pi_A\colon X_A\to \XX$ extending the projection $\pi^{H_A,\widehat H_A}\colon A/H_A\to A/\widehat H_A$. The $A$-colors and the $G$-colors of $\XX$ coincide, and this implies the same for $X_A$: indeed any $A$-color (resp.\ $G$-color) of $X_A$ is of the form $\pi_A^{-1}(D)$ for an $A$-color (resp.\  $G$-color) $D$ of $\XX$.

If $D\subset \XX$ is a color such that $\pi_A^{-1}(D)$ contains a $G$-orbit $Y\subset X_A$, then $D$ contains the $G$-orbit $\pi_A(Y)$: this is absurd because $\XX$ is a toroidal $G$-variety. In other words $X_A$ is a toroidal $G$-variety.

Next, we claim that $A/H_A$ is a $G$-embedding of $G/H$ whose fan contains $c_{X,E}$ as its unique non-trivial cone. Part (\ref{lemma:restrictions:openorbit}) of Lemma \ref{lemma:restrictions} implies that $A/\widehat H_A$ is an elementary embedding of $G/\widehat H$, with orbits $G/\widehat H$, $\pi(E)\cap A/\widehat H_A$, and fan containing $c_{\XX,\pi(E)}$ as its unique non-trivial cone. The open subset $G/H\subset A/H_A (\subseteq X_A)$ is equal to $\pi_A^{-1}(G/\widehat H)$, and the $G$-stable closed subset $E'=(A/H_A)\smallsetminus(G/H)$ is equal to $\pi_A^{-1}(\pi(E))\cap A/H_A$.

Consider the $G$-invariant prime divisors contained in $E'$: they are neither colors nor $A$-stable prime divisors. We claim that there is only one of them, with associated convex cone $c_{X,E}$. Then $E'$ itself is a $G$-stable prime divisor, because we already proved that $A/H_A$ is a toroidal embedding of $G/H$.

To show the claim, consider $f\in\CC(G/H)^{(B)}_\lambda$ with $\lambda\in\Lambda$. By Lemma \ref{lemma:restrictions2} we have that $f$ is also a $B_A$-eigenvector, therefore its divisor $\div(f)$ on $A/H_A$ has components which are either colors or $A$-stable prime divisors. It follows that $\rho_{G,A/H_A}(F)\in\lambda^\perp$ for all $\lambda\in\Lambda$ and all $G$-stable prime divisor $F\subseteq E'$. Since $c_{X,E}= \Lambda^\perp\cap\V_G(G/H)$, we deduce that there is only one such $F$ and it satisfies $\rho_{G,A/H_A}(F)\in c_{X,E}$: the claim above follows.

Now Lemma~\ref{lemma:movedareorthogonal}, Lemma~\ref{lemma:restrictions2} and Corollary~\ref{cor:fan} part (\ref{cor:fan:sigma}) hold also if we replace $X$ with $X_A$ and $\D$ with the set $(\partial_G X_A)\smallsetminus\{ E' \}$. From Corollary~\ref{cor:fan} part (\ref{cor:fan:sigma}) we deduce that $\F_G(X_A)$ is the join of $\F_G^{\sigma_{\pi(E)}}(X_A)$ and $c_{X,E}$. From Lemma~\ref{lemma:restrictions2} we deduce that every $G$-stable prime divisor $D$ of $X_A$ such that $\rho_{G,X_A}(D)\in\sigma_{\pi(E)}^\perp$ is also $A$-stable, hence each $G$-orbit $Y\subseteq X_A$ such that $c_{X_A,Y}\subset\sigma_{\pi(E)}^\perp$ is also an $A$-orbit.

In other words $\F_G^{\sigma_{\pi(E)}}(X_A)$ and $\F_G^{\sigma_{\pi(E)}}(X)$ have the same image under $s_*$, which implies that they are equal. The theorem in the case $|E|=1$ follows.

If $|\E|>1$, we consider the chain of groups
\[
\theta_{G,X}(G)\subseteq \Autz(X,\partial_G X\smallsetminus\{E_1\})\subseteq\Autz(X,\partial_G X\smallsetminus\{E_1,E_2\})\subseteq\ldots\subseteq\Autz(X,\D),
\]
where $\E=\{E_1,E_2,\ldots\}$, and proceed by induction on $|\E|$. Let $A_i\subseteq A_{i+1}$ be two consecutive groups of this chain: we may apply the first part of the proof, together with Corollary~\ref{cor:partialt} below (whose proof in the case $|\E|=1$ only depends on the case $|\E|=1$ of this theorem) to the $A_{i}$-variety $X$. We obtain the construction of an $A_{i+1}$-variety $X_{A_{i+1}}$, which is $A_i$-equivariantly isomorphic to $X$.
\end{proof}

\begin{corollary}\label{cor:nonlinear:image}
The image of $A$ in $\Autz(X)$ is equal to $\Autz(X,\D)$.
\end{corollary}
\begin{proof}
By construction $A$ moves each element of $\EE$ on $\XX$ and stabilizes all elements of $\DD$, hence $\DD=\partial_A\XX$.

Moreover $\widehat H_A$ is the wonderful closure of $H_A$, hence we can apply the exact sequence (\ref{eqn:partial}) to $X$ as an $A$-variety, mapping onto the wonderful $A$-variety $\XX$. Since the image of $A$ contains by construction both the universal cover of $\Autz(\XX,\DD)=\Autz(\XX,\partial_A\XX)$ and $(\widehat H_A/H_A)^\circ\cong(\widehat H/H)^\circ \subseteq C$, it follows that the image of $A$ contains $\Autz(X,\partial_AX)=\Autz(X, \D)$.
\end{proof}

We recall that our assumptions on $\D$ include the fact that all elements of $\E=\partial X\smallsetminus \D$ are not stable under the action of $\Autz(X)$. From the above theorem, together with Corollary~\ref{cor:nonlinear:image} and the definition of $\F_A$, it follows that $\partial_AX=\D$, in other words all elements of $\E$ are also not stable under the action of $\Autz(X,\D)$. We end this section with the following additional observation that will be useful later.

\begin{corollary}\label{cor:partialt}
We have $(\partial_A X)^\ell=(\partial_G X)^\ell$.
\end{corollary}
\begin{proof}
This is obvious from the definition of $\F_A$.
\end{proof}

\section{Abelian case}\label{s:abelian}
In this section we will assume that $G=C$ is an algebraic torus, $X$ as usual a complete $G$-regular variety, and $\D\subseteq \partial_G X$ any subset. Hence $X$ is a toric variety under the acton of a quotient of $G$. Since $G$ is its own Borel subgroup, $X$ has no $G$-color. In this setting the study of $\Autz(X)$ is simplified by the fact that, for all $D\in\partial_G X$, the $G$-module $H^0(X, \OO_X(D))$ splits into the sum of $1$-dimensional $G$-submodules.

We report some results on $\Autz(X)$ from \cite{De70}, starting with the following proposition on one-dimensional subgroups of $\Autz(X)$. A standard reference on this subject is also \cite{Oda88}.

Let us recall that any element of $\Hom_\ZZ(\Lambda_G(X),\ZZ)=\Hom_\ZZ(\Chi(\theta_{G,X}(G)),\ZZ)$ is canonically associated with a one-parameter subgroup $\GG_m\to G$. This induces a tangent vector field on $X$ via the action of $G$, namely the image, via the differential of $\GG_m\to \Autz(X)$, of the derivation of $\CC[\GG_m]\cong \CC[z,z^{-1}]$ defined by $z\mapsto z$.

\begin{proposition}[{\cite[Section 3.4]{Oda88}}] \label{prop:oda}
Let $\alpha\in\Lambda_G(X)$ be non-zero. The divisor $X(\alpha)$ exists if and only if there exists a one-dimensional unipotent subgroup $U_\alpha\subset \Autz(X)$ such that $G$ normalizes $U_\alpha$, and the action of $G$ by conjugation on $U_\alpha\cong \CC$ is linear with weight $\alpha$. If this is the case, then the following holds.
\begin{enumerate}
\item The subgroup $U_\alpha\subset \Autz(X)$ is unique, and $X(\alpha)$ is the unique $G$-stable prime divisor of $X$ not stable under the action of $U_\alpha$.
\item Let $c_\alpha$ be the one dimensional cone of the fan $\F_G(X)$ associated with $X(\alpha)$, let $X_0$ be a $G$-stable affine open subset of $X$ and let $c$ be the cone in $\F_G(X)$ corresponding to the closed $G$-orbit of $X_0$. If $c_\alpha$ is an edge of $c$, then $X_0$ is $G\ltimes U_\alpha$-stable.
\item Fix an isomorphism $u_\alpha\colon\CC\to U_\alpha\subseteq\Autz(X)$ and consider the differential $du_\alpha\colon \Lie\CC\to \Lie\Autz(X)$. Then $du_\alpha(d/d\xi)= f_\alpha \delta_\alpha$, where $d/d\xi$ is the derivation with respect to the coordinate $\xi$ of $\CC$, the function $f_\alpha$ is in $\CC(X)^{(B)}_\alpha$, and $\delta_\alpha$ is the tangent vector field on $X$ induced by $\rho_{G,X}(X(\alpha))$, seen as an element of $\Hom_\ZZ(\Chi(\theta_{G,X}(G)),\ZZ)$.
\item More explicitly, if $f_\beta \in \CC(X)^{(G)}_\beta$ for some $\beta\in\Lambda_G(X)$ and $\xi \in U_\alpha\cong\CC$, then the rational function $z\mapsto f_\beta(\xi \cdot z)$ on $X$ is given by the formula
\begin{equation}	\label{eqn:oda-azione}
f_\beta(\xi \cdot z) = f_\beta(z) (1+\xi f_\alpha(z))^{\langle \rho_{G,X}(X(\alpha)), \beta \rangle},
\end{equation}
(all the $G$-semiinvariant rational functions are normalized here in such a way that they take value $1$ on the same element in the open $G$-orbit of $X$).
\end{enumerate}
\end{proposition}

\begin{remark}
If $X(\alpha)$ exists for some $\alpha$ then we have $\langle \rho_{G,X}(X(\alpha)), \alpha\rangle = -1$ and $\langle \rho_{G,X}(D), \alpha\rangle \geq 0$ for all $D\in\partial_GX$ different from $X(\alpha)$. However, the difference in signs from our discussion (in particular in Proposition~\ref{prop:oda}) and \cite[Section 3.4]{Oda88} is only apparent: a character $\lambda\in\Chi(\theta_{G,X}(G))$ is indeed a rational function on $X$ and a $G$-eigenvector, but of $G$-eigenvalue $-\lambda$.
\end{remark}

Notice that the assignment $\alpha\mapsto X(\alpha)$ might be not injective. Also, if both $X(\alpha)$ and $X(-\alpha)$ exist, then $\rho_{G,X}(X(\alpha))$ is not necessarily $-\rho_{G,X}(X(-\alpha))$. However, $X(\alpha)$ and $X(-\alpha)$ are the only $G$-stable prime divisors whose images through $\rho_{G,X}$ are non-zero on $\alpha$.

\begin{definition}
Let $\D\subseteq \partial_G X$ any subset, and define $\Phi=\Phi(X,\D)$ to be the maximal set of roots of $X$ such that:
\begin{enumerate}
\item if $\alpha\in\Phi(X,\D)$ then also $-\alpha\in\Phi(X,\D)$;
\item if $\alpha\in\Phi(X,\D)$ then $X(\alpha)\in\E=\partial X\smallsetminus \D$.
\end{enumerate}
\end{definition}

The following result is a consequence of \cite[Demazure's Structure Theorem, Section 3.4]{Oda88}.

\begin{theorem}
The subgroup of $\Autz(X)$ generated by $\theta_{G,X}(G)$ and $U_\alpha$ for all $\alpha\in\Phi(X,\D)$ has $\Phi(X,\D)$ as root system with respect to its maximal torus $\theta_{G,X}(G)$, and is a Levi subgroup of $\Autz(X,\D)$.
\end{theorem}

\begin{example}\label{ex:Pn}
Let $X=\PP^n$ with $n\geq 2$, under the linear action of the group $G$ of $(n+1)\times (n+1)$ invertible diagonal matrices. Then $\partial_GX$ has $n+1$ elements $D_1,\ldots,D_{n+1}$ given resp.\ by the vanishing of the homogeneous coordinates $x_1,\ldots,x_{n+1}$. The lattice $\Lambda_G(X)$ is the root lattice of the group $\GL(n+1)\supset G$; the simple roots $\alpha_1,\ldots,\alpha_n$ (with respect to the Borel subgroup of upper triangular matrices and its maximal torus $G$, and in the usual ordering) are the $G$-eigenvalues of the $G$-semiinvariant rational functions $x_2/x_1,x_3/x_2,\ldots,x_{n+1}/x_n$. In the basis $\alpha_1^*,\ldots,\alpha_n^*$ of $\N_G(X)$ dual to $\alpha_1,\ldots,\alpha_n$, the $n+1$ one-dimensional cones in the fan $\F_G(X)$ are generated by $\rho(D_1)=-\alpha_1^*,\rho(D_2)=\alpha_1^*-\alpha_2^*,\ldots,\rho(D_n)=\alpha_{n-1}^*-\alpha_n^*,\rho(D_{n+1})=\alpha_n^*$.

The set of all roots of $X$ as a toric variety is the set of roots of $\GL(n+1)$, and the roots $\alpha$ such that $X(\alpha)=D_{n+1}$ are the negative roots with $\alpha_n$ in their support. Correspondingly, if $\D=\{D_{n+1}\}$ then $\Autz(X,\D)$ is a maximal parabolic subgroup of $\PGL(n+1)$, its Levi subgroup $A$ containing the image of $G$ is the subgroup whose roots don't have $\alpha_n$ in their support.

Notice that $\Autz(X,\D)$ acts transitively on $\PP^n\smallsetminus D_{n+1}\cong \CC^n$. A one-dimensional ``root subgroup'' $U_\alpha$ contained in the unipotent radical of $\Autz(X,\D)$ acts as the translations along one of the coordinates, according to formula (\ref{eqn:oda-azione}).
\end{example}

\begin{definition}
Define $A=A(X,\D)$ the subgroup of $\Autz(X)$ generated by $\theta_{G,X}(G)$ and $U_\alpha$ for all $\alpha\in\Phi(X,\D)$. Let us also choose a Borel subgroup $B_A\subseteq A$ containing $G$ and, consequently, a subdivision of $\Phi$ into positive and negative roots, resp.\ denoted by $\Phi_+=\Phi_+(X,\D)$ and $\Phi_-=\Phi_-(X,\D)$, and denote by $\Psi=\Psi(X,\D)$ the basis of positive roots.
\end{definition}

Since $B_A$ is generated by $\theta_{G,X}(G)$ together with the subgroups $U_{\alpha}$ for all $\alpha\in\Psi$, we have that any $G$-stable prime divisor which doesn't appear as $X(\alpha)$ for some $\alpha\in\Psi$ is $B_A$-stable. In other words
\begin{equation}\label{eqn:psi+}
\left\{ X(\alpha)\;\middle\vert\;\alpha\in\Phi_+ \right\} = \left\{ X(\alpha)\;\middle\vert\;\alpha\in\Psi\right\},
\end{equation}
and for the same reason (replacing $\Psi$ with $-\Psi$)
\begin{equation}\label{eqn:psi-}
\left\{ X(\alpha)\;\middle\vert\;\alpha\in\Phi_- \right\} = \left\{ X(\alpha)\;\middle\vert\;\alpha\in(-\Psi)\right\}.
\end{equation}

\begin{lemma}\label{lemma:triple}
Let $\alpha,\beta\in\Phi$ be different, and suppose that $X(\alpha) = X(\beta)$. Then $\gamma= \alpha - \beta$ and $-\gamma$ are also in $\Phi$, with $X(\gamma)=X(-\beta)$ and $X(-\gamma)=X(-\alpha)$.
\end{lemma}
\begin{proof}
Suppose that $X(-\alpha)=X(-\beta)$. Then $\alpha-\beta$ is zero on $\rho_{G,X}(X(\pm\alpha))$ and on $\rho_{G,X}(X(\pm\beta))$. On the other hand, if a $G$-stable prime divisor $D\subset X$ is not of the form $X(\pm\alpha)$ nor $X(\pm\beta)$, then both $\alpha$ and $\beta$ are zero on $\rho_{G,X}(D)$. It follows that $\supp\F_G(X)$ is contained in the hyperplane $(\alpha-\beta)^\perp$ of $\N_G(X)$, which contradicts the completeness of $X$. Therefore $X(-\alpha)\neq X(-\beta)$, i.e.\ $X(\alpha)$, $X(-\alpha)$ and $X(-\beta)$ are three different prime divisors. The statement of the lemma is now obvious.
\end{proof}

\begin{lemma}\label{lemma:independent}
The matrix
\begin{equation}\label{eqn:mat}
\left(\langle \rho_{G,X}(X(\alpha)), \alpha'\rangle\right)_{\alpha,\alpha'\in\Psi}
\end{equation}
is non-degenerate. In particular, the elements $\rho_{G,X}(X(\alpha))$, for $\alpha$ varying in $\Psi$, are linearly independent.
\end{lemma}
\begin{proof}
Thanks to Lemma~\ref{lemma:triple}, the elements $\rho_{G,X}(X(\alpha))$ for $\alpha\in\Psi$ are pairwise distinct. If the matrix (\ref{eqn:mat}) is degenerate, there exists a linear combination
\begin{equation}\label{eqn:dep}
\sum_{\alpha\in\Psi'} a_\alpha \rho_{G,X}(X(\alpha)) \in\Psi^\perp
\end{equation}
where $\varnothing\neq\Psi'\subseteq \Psi$ and $a_\alpha\neq0$ for all $\alpha\in\Psi'$. Applying $\langle\alpha,-\rangle$ for a fixed $\alpha\in\Psi'$ to the linear combination (\ref{eqn:dep}), we see that both $\rho_{G,X}(X(\alpha))$ and $\rho_{G,X}(X(-\alpha))$ must appear in the sum. Indeed, the former appears, and the latter is the only other possible summand that is nonzero on $\alpha$. The elements $\rho_{G,X}(X(-\alpha))$ for $\alpha\in\Psi$ are distinct, thanks to the first part of the proof applied to the set of simple roots $-\Psi$.

Hence each summand in (\ref{eqn:dep}) can also be rewritten as $a_\alpha \rho_{G,X}(X(-\tau(\alpha)))$ where $\tau\colon\Psi'\to\Psi'$ is a bijection. We also know that $\rho_{G,X}(X(\alpha))\neq\rho_{G,X}(X(-\alpha))$, therefore $\tau$ has no fixed points. Now consider
\[
\gamma = \sum_{\alpha\in\Psi'} \alpha.
\]
Its value on $\rho_{G,X}(D)$ is zero, if $D\subset X$ is a $G$-stable prime divisor not of the form $X(\pm\alpha)$ for some $\alpha\in\Psi'$. On the other hand, for a fixed $\alpha\in\Psi'$ we have that $X(\alpha)=X(-\tau(\alpha))$, but $X(\alpha)\neq X(\beta)$ for all $\beta\in\Psi$ different from $\alpha$, and $X(\alpha)\neq X(-\beta)$ for any $\beta\in\Psi$ different from $\tau(\alpha)$. Therefore
\[
\begin{array}{rcl}
\langle \rho_{G,X}(X(\alpha)), \gamma\rangle & = & \langle \rho_{G,X}(X(\alpha)), \alpha\rangle + \langle\rho_{G,X}(X(\alpha)),\tau(\alpha)\rangle + \\ &  & + \left\langle \rho_{G,X}(X(\alpha)), {\displaystyle \sum_{\beta\in\Psi',\beta\neq\alpha,\tau(\alpha)}\beta} \right\rangle \\ & = & -1+1+0=0.
\end{array}
\]
We obtain that $\supp\F_G(X)$ is contained in the hyperplane $\gamma^\perp$, which is absurd because $X$ is complete.
\end{proof}

\begin{proposition}\label{prop:abelian}
As an $A$-variety, $X$ is spherical (not necessarily toroidal). The set of its $A$-stable prime divisors is
\[
\partial_AX=\partial_G X\smallsetminus \{ X(\alpha)\;\vert\; \alpha\in\Phi \},
\]
and these are exactly the $G$-stable prime divisors $D$ such that $\rho_{G,X}(D)\in\Psi^\perp$. Given the identification $\Chi(\theta_{G,X}(G))=\Chi(B_A)$, we have an inclusion
\[
\iota\colon \Lambda_A(X) \to \Lambda_G(X)
\]
whose image is the sublattice
\begin{equation}\label{eqn:image}
\left\{ \rho_{G,X}(X(\alpha))\;\middle\vert\;\alpha\in\Psi \right\}^\perp \subseteq \Lambda(X).
\end{equation}
The restriction map $\iota^*\colon\N_G(X)\to \N_A(X)$ induces an isomorphism
\[
\iota^*|_{\Psi^\perp}\colon \Psi^\perp \stackrel{\cong}{\to} \N_A(X).
\]
For any $B_A$-stable prime divisor $D\subset X$ we have $\rho_{A,X}(D) = \iota^*\rho_{G,X}(D)$, and the set of $A$-colors of $X$ is the following:
\[
\Delta_A(X) = \left\{ X(\alpha) \;\middle\vert\; \alpha\in(-\Psi) \right\} \smallsetminus \left\{ X(\alpha) \;\middle\vert\; \alpha\in\Psi \right\}. 
\]
The set of simple roots associated to $\P_A(X)$ is
\[
\{ \alpha\in \Psi \;|\; X(-\alpha)\notin \Delta_A(X)\}.
\]
Finally, let $\alpha\in\Psi$ with $X(-\alpha)\in\Delta_A(X)$. For all $\beta\in\Phi_+$ different from $\alpha$, we have $X(-\alpha)\neq X(\beta)$ and $X(\alpha)\neq X(\beta)$. In particular, if in addition $\beta\in\Psi$, we also have $\rho_{G,X}(X(-\alpha))\in \beta^\perp$.
\end{proposition}
\begin{proof}
Since $\theta_{G,X}(G)\subseteq B_A$ has already an open orbit on $X$, the first statement is obvious. The statement about the $A$-stable prime divisors is also immediate.

Let us prove that the $A$-colors are the set $\Delta_A(X)$ as above defined. A color must be $X(\alpha)$ for some $\alpha\in\Phi$ otherwise it is $A$-stable, and at this point not being of the form $X(\alpha)$ for any $\alpha\in\Phi_+$ is equivalent to be stable under $B_A$. Then, we conclude using (\ref{eqn:psi+}) and (\ref{eqn:psi-}). The statement on the parabolic subgroup $\P_A(X)$ is obvious.

The inclusion $\iota$ is given by the simple observation that a $B_A$-eigenvector in $\CC(X)$ is a fortiori a $G$-eigenvector, with same eigenvalue; the identity $\rho_{A,X}(D)=\iota^*\rho_{G,X}(D)$ for any $B_A$-stable prime divisor is also obvious.

Let us prove that the image of $\iota$ is the lattice (\ref{eqn:image}). If $\gamma\in\Lambda_A(X)$, then a corresponding $B_A$-eigenvector $f_\gamma\in\CC(X)$ cannot have zeros nor poles on prime divisors $X(\alpha)$ for $\alpha\in\Psi$, since the latter divisors are not $B_A$-stable. Hence $\iota(\Lambda_A(X))\subseteq\left\{ \rho_{G,X}(X(\alpha))\;\middle\vert\;\alpha\in\Psi \right\}^\perp$. On the other hand, if $\chi\in\left\{ \rho_{G,X}(X(\alpha))\;\middle\vert\;\alpha\in\Psi \right\}^\perp$, then a corresponding $G$-eigenvector $f_\chi\in\CC(X)$ has zeros and poles only on $A$-stable prime divisors or on colors. It follows that $f_\chi$ is also a $B_A$-eigenvector, and the other inclusion is proved.

We prove now that $\iota^*|_{\Psi^\perp}$ is an isomorphism between $\Psi^\perp$ and $\N_A(X)$. From the first part of the proof, this follows if we prove that
\[
\Lambda\otimes_\ZZ\QQ = \left(\Psi\otimes_\ZZ\QQ\right) \oplus \left(\left\{ \rho_{G,X}(X(\alpha))\;\middle\vert\;\alpha\in\Psi \right\}^\perp\otimes_\ZZ\QQ\right),
\]
and this equality is a consequence of Lemma~\ref{lemma:independent}.

Let us check the last statement, so let $\alpha\in\Psi$ be such that $X(-\alpha)\in\Delta_A(X)$, and consider $\beta\in\Phi_+$, $\beta\neq\alpha$. We know that $X(-\alpha)\neq X(\beta)$ because of the definition of $\Delta_A(X)$ together with (\ref{eqn:psi+}). This also implies that $X(\alpha)\neq X(\beta)$, because otherwise we would have $\beta-\alpha\in\Phi_+$ with $X(-\alpha)=X(\beta-\alpha)$, thanks to Lemma~\ref{lemma:triple}.
\end{proof}

\begin{remark}
The two above results imply in particular that the $A$-colors of $X$, seen as elements of $\N_A(X)$, are linearly independent.
\end{remark}

\begin{example}
The projective space $X=\PP^n$, with $\D$ as in Example~\ref{ex:Pn} and $n\geq 2$, is not toroidal as an $A$-variety. The group $A\cong\GL(n)$ acts with a fixed point contained in all elements of $\partial_GX\smallsetminus\D$, therefore contained in all $A$-colors of $X$.
\end{example}

We can now state the main theorem of this section.

\begin{theorem}\label{thm:abelian}
If we identify $\N_A(X)$ and $\Psi^\perp$ via the map $\iota^*|_{\Psi^\perp}$ of Proposition~\ref{prop:abelian}, the fan of colored convex cones $\F_A(X)$ of $X$ as a spherical $A$-variety is obtained from the fan $\F_G(X)$ as follows:
\[
\F_A(X) = \left\{ (c\cap\Psi^\perp ,d(c)) \;\middle\vert\; c \in\F_G(X) \right\}.
\]
Here $d(c)$ is the set of $A$-colors $D$ of $X$ such that if $\beta\in\Phi_+$  satisfies $X(-\beta)=D$, then both $\rho_{G,X}(X(\beta))$ and $\rho_{G,X}(X(-\beta))$ lie on $1$-dimensional faces of $c$.
\end{theorem}
\begin{proof}
First, we consider $c\in\F_G(X)$ and we show that the colored cone $(c\cap\Psi^\perp,d(c))$ belongs to $\F_A(X)$.

The cone $c$ is equal to $c_{X,Y}$ for some $G$-orbit $Y$. We claim that the colored cone associated to the $A$-orbit $AY$ is given by $(c\cap\Psi^\perp, d(c))$, with $d(c)$ defined as in the theorem. To show the claim, it is enough to prove that:
\begin{enumerate}
\item \label{proof:abelian:A} the $A$-stable prime divisors containing $AY$ are the $G$-stable prime divisors $D$ such that $D\supseteq Y$ and $\rho_{G,X}(D)\in\Psi^\perp$;
\item \label{proof:abelian:colors} the set of the $A$-colors containing $AY$ is $d(c)$;
\item \label{proof:abelian:cone} the convex cone $c'$ generated by the image of elements of (\ref{proof:abelian:A}) and (\ref{proof:abelian:colors}) under the map $\iota^*\circ\rho$ is $c\cap\Psi^\perp$.
\end{enumerate}

Part (\ref{proof:abelian:A}) is obvious, thanks to the results on $\partial_AX$ contained in Proposition~\ref{prop:abelian}. For part (\ref{proof:abelian:colors}), let us first prove that a color $D$ not belonging to $d(c)$ doesn't contain $AY$. If $D$ doesn't contain $Y$ there is nothing to prove, therefore we may assume that $\rho_{G,X}(D)$ lies on a $1$-dimensional face of $c$. Suppose at first that $X(-\beta)=D$ for some $\beta\in\Phi_+$, in such a way that $X(\beta)$ doesn't contain $Y$.

Let $X_c$ be the affine $G$-stable open subset of $X$ such that its unique closed $G$-orbit has cone $c$. In other words:
\[
X_c = \left\{ x\in X \;\middle\vert\; \overline{Gx}\supseteq Y \right\}.
\]
It is isomorphic to an affine space, and thanks to Proposition~\ref{prop:oda} it is stable under the action of $U_{-\beta}$.

More precisely, there exist global coordinates $(x_1,\ldots,x_n)$ on $X_c$ such that $X(-\beta)\cap X_c$ is the hyperplane defined by the equation $x_1=0$, and in these coordinates $U_{-\beta}$ acts as follows:
\begin{equation}\label{eqn:color1}
u_{-\beta}(\xi)(x_1,x_2,\ldots,x_n) = (x_1+\xi,x_2,\ldots,x_n).
\end{equation}
This formula is checked directly using Proposition~\ref{prop:oda} and the fact that $X(-\beta)$ is the only $G$-stable prime divisor that contains $Y$ and where $\beta$ is non-zero. The hyperplane defined in $X_c$ by $x_1=0$ contains $Y$, but from (\ref{eqn:color1}) we deduce that it doesn't contain $U_\beta Y$. As a consequence, $AY$ is not contained in $X(-\beta)$.

Now we show that a color $D$ in $d(c)$ contains $AY$. At first, consider $\beta\in\Phi_+$ such that $X(-\beta)=D$. Both $X(-\beta)$ and $X(\beta)$ contain $Y$, and we consider again the affine space $X_c$. Applying Proposition~\ref{prop:oda} once again, there exist coordinates $(x_1,x_2,\ldots,x_n)$ such that $X(-\beta)\cap X_c$ is defined by the equation $x_1 = 0$, and $X(\beta)\cap X_c$ by the equation $x_2= 0$, and such that
\begin{equation}\label{eqn:color2a}
u_{-\beta}(\xi)(x_1,x_2,\ldots,x_n) = (x_1+\xi x_2,x_2,\ldots,x_n)
\end{equation}
and
\begin{equation}\label{eqn:color2b}
u_{\beta}(\xi)(x_1,x_2,\ldots,x_n) = (x_1,x_2+\xi x_1,\ldots,x_n).
\end{equation}
We obtain that $Y$ is both $U_{-\beta}$-stable and $U_{\beta}$-stable, being contained in the subset of $X_c$ defined by $x_1=x_2=0$. Therefore $X(-\beta)=D$ contains $Y=U_\beta U_{-\beta} Y$.

Now observe that the image of the multiplication map
\[
\prod_{\gamma\in\Phi}U_\gamma\times\theta_{G,X}(G) \to A
\]
(where the product is taken in any fixed order) is dense in $A$. It follows that $D$ contains $AY$, if we prove that $D$ is $U_\gamma$-stable for all $\gamma\in\Phi$ such that $\gamma\neq\pm\beta$ for all $\beta\in\Phi_+$ satisfying $X(-\beta)=D$. For $\gamma\in\Phi_-$ there is nothing to prove. But also for $\gamma\in\Phi_+$ we know that $D\neq X(\gamma)$: this fact stems from the last statement of Proposition~\ref{prop:abelian} together with (\ref{eqn:psi+}). The proof of (\ref{proof:abelian:colors}) is complete.

Let us prove (\ref{proof:abelian:cone}). Call $P$ the set of $A$-stable prime divisors containing $AY$. Then we can describe a minimal set of generators of $c$ (as a convex cone) as the union of the following subsets:
\begin{itemize}
\item[(a)] the set $\rho_{G,X}(P)$;
\item[(b)] for each color $D\in d(c)$ the set
\[
\{\rho_{G,X}(D)\} \cup \{ \rho_{G,X}(X(\beta)) \;|\; \beta\in\Phi_+,\, X(-\beta)=D \};
\]
\item[(c)] other generators, different from any of the above.
\end{itemize}

We show that $c\cap\Psi^\perp$ is contained in $c'$, and recall that the latter is generated by $\rho_{G,X}(P)$ together with $\iota^*(\rho_{G,X}(d(c)))$. An element $x\in c\cap\Psi^\perp$ is a linear combination with non-negative coefficients of the above generators, and we may assume that the elements of (a) do not contribute. This indeed implies the general case, since $\rho_{G,X}(P)\subseteq c'$.

Also, we may suppose that any generator $z$ involved in the linear combination giving $x$ satisfies $\iota^*(z)\neq0$. Indeed, otherwise we may suppress it using the fact that $x=\iota^*(x)$. Hence, all generators in the linear combination of $x$ are not of the form $\rho_{G,X}(X(\beta))$ for $\beta\in\Psi$.

It remains the generators $\rho_{G,X}(D)$ where $D\in d(c)$, and generators of (c) of the form $\rho_{G,X}(X(-\alpha))$ for some $\alpha\in\Psi$. In the second case $X(-\alpha)$ is a color, because it cannot be equal to $X(\beta)$ for any $\beta\in\Psi$. Being not in $d(c)$, each such $X(-\alpha)$ admits a positive root $\beta$ satisfying $X(-\beta) =  X(-\alpha)$ and $D_\beta$ not a generator of $c$. This implies that $\beta$ is non-positive on $c$, and the only chance for $x$ to be in $\beta^\perp$ is that such a generator $X(-\beta) = X(-\alpha)$ doesn't occur.

As a consequence, $x$ is a linear combination of the elements $\rho_{G,X}(D)$ with $D\in d(c)$, and we conclude that $x\in c'$ using again $\iota^*(x)=x$.

Finally, let $x\in c'$, and let us show that $x\in c\cap\Psi^\perp$. As before, we ignore the generators of $c'$ lying in $\Psi^\perp$, and we assume that $x$ is a linear combination with non-negative coefficients of $\iota^*(d(c))$. In other words:
\[
x = \sum_{\alpha\in\Psi, X(-\alpha)\in d(c)} a_\alpha \iota^*\left(\rho_{G,X}\left(X(-\alpha)\right)\right)
\]
with $a_\alpha \geq 0$. Consider a summand $a_\alpha \iota^*\left(\rho_{G,X}\left(X(-\alpha)\right)\right)$. For each positive root $\beta\neq\alpha$ such that $X(-\beta)=X(-\alpha)$, Lemma~\ref{lemma:triple} implies that $\gamma=\beta-\alpha$ and $-\gamma$ are also roots in $\Phi$, and that $X(-\alpha)=X(-\beta)$, $X(\alpha) = X(-\gamma)$, $X(\gamma) = X(\beta)$ are three distinct prime divisors. Then, we take the sum
\begin{equation}\label{eqn:y}
y = \sum_{\alpha\in\Psi, X(-\alpha)\in d(c)} a_\alpha y_\alpha
\end{equation}
where
\[
y_\alpha = \rho_{G,X}\left(X(-\alpha)\right)+ \sum_{\substack
{\beta\in\Phi_+, \\ X(-\beta)=X(-\alpha)}} \rho_{G,X}\left(X(\beta)\right).
\]
We claim that all simple roots in $\Psi$ are zero on this element, hence $\iota^*(y)=y$ and we immediately conclude that $y=x$. On the other hand, $y$ is in $c$ thanks to the definition of the set $d(c)$, therefore $x\in c\cap\Psi^\perp$.

Let us prove the claim. Let $\gamma\in\Psi$, and pick a $y_\alpha$. If $\gamma=\alpha$, then the last assertion of Proposition~\ref{prop:abelian} implies that $y_\alpha$ is the sum of $\rho_{G,X}(X(-\alpha))$ and $\rho_{G,X}(X(\alpha))$, plus other terms where $\alpha$ is zero. It follows $\langle y_\alpha, \gamma \rangle =0$.

If $\gamma\neq\alpha$, then $\langle\rho_{G,X}(X(-\alpha)), \gamma\rangle= 0$ thanks to Proposition~\ref{prop:abelian}. Moreover, in this case $\gamma$ does not appear as a $\beta$ in the sum expressing $y_\alpha$, because we know that $X(-\alpha) \neq X(-\gamma)$. Also, if $X(\pm\gamma)$ is different from $\rho_{G,X}(X(\beta))$ for all $\beta\in\Phi_+$ such that $X(-\beta)=X(-\alpha)$, then again $\langle y_\alpha, \gamma\rangle =0$.

Therefore we may suppose that $\gamma$ is different from all the $\beta$ appearing in the expression of $y_\alpha$, but some of them, say $\beta_{i,\gamma}$ for $i=1,\ldots,k$, satisfy $X(\beta_{i,\gamma}) = X(\epsilon_{i,\gamma}\gamma)$ where $\epsilon_{i,\gamma}=1$ or $-1$. In this case Lemma~\ref{lemma:triple} implies that $\beta_{i,\gamma} - \epsilon_{i,\gamma}\gamma$ also appears in the sum, with $X(\beta_{i,\gamma} - \epsilon_{i,\gamma}\gamma) = X(-\epsilon_{i,\gamma}\gamma)$. We obtain:
\[
\begin{array}{rcl}
y_\alpha &=& {\displaystyle\rho_{G,X}\left(X(-\alpha)\right)+ \sum_{i=1}^k\left( \rho_{G,X}\left(X(\beta_{i,\gamma})\right) + \rho_{G,X}\left(X(\beta_{i,\gamma} - \epsilon_{i,\gamma}\gamma)\right)\right) +} \\
& & + {\displaystyle \sum_{\substack{\beta\in\Phi_+, X(\beta)\neq X(\pm\gamma)\\ X(-\beta)=X(-\alpha)}} \rho_{G,X}\left(X(\beta)\right)} \\
&=& {\displaystyle\rho_{G,X}\left(X(-\alpha)\right)+ \sum_{i=1}^k\left( \rho_{G,X}\left(X(\gamma)\right) + \rho_{G,X}\left(X(-\gamma)\right)\right) + } \\
& & {\displaystyle+ \sum_{\substack{\beta\in\Phi_+, X(\beta)\neq X(\pm\gamma)\\ X(-\beta)=X(-\alpha)}} \rho_{G,X}\left(X(\beta)\right)}.
\end{array}
\]
From this expression it is evident that $\langle y_\alpha,\gamma\rangle =0$, and the proof of (\ref{proof:abelian:cone}) is complete.

To finish the proof of the theorem, we must check that all colored cones of $\F_A(X)$ appear as $(c\cap\Psi^\perp,d(c))$ for some $c\in\F_G(X)$. For this, it is enough to notice that for each $A$-orbit $Z$ there is a $G$-orbit $Y$ such that $AY=Z$.
\end{proof}

\begin{corollary}\label{cor:abelianSigma}
The $A$-variety is horospherical, i.e.\ $\Sigma_A(X)=\varnothing$.
\end{corollary}
\begin{proof}
There exists a smooth complete toroidal $A$-variety $Y$ equipped with a surjective birational $A$-equivariant morphism $Y\to X$ (it is enough to choose an $A$-equivariant resolution of singularities of the variety given in \cite[Lemma 5.2]{Kn91}, where $X''$ in the proof of {\em loc.cit.} is our $X$).

Then $\Sigma_A(Y)=\Sigma_A(X)$, and $Y$ is also a complete $G$-regular  embedding. Applying Theorem~\ref{thm:abelian} to $Y$, it follows that $\supp\F_A(Y)$ is a vector space, and it is equal to $\V_A(Y)$ because $Y$ is toroidal and complete. We conclude that $\Sigma_A(Y)=\varnothing$.
\end{proof}

\begin{remark}
With a slightly more involved proof, one can derive the above corollary directly from Proposition~\ref{prop:abelian} and avoid using Theorem~\ref{thm:abelian}.
\end{remark}

\begin{remark}
Notice that $d(c)=\varnothing$ if and only if $c\cap\Psi^\perp$ is a face of $c$. In other words $X$ is toroidal under the action of $A$ if and only if the subspace $\Psi^\perp$ intersects each cone of the fan of $X$ in a face of the cone.
\end{remark}

\begin{example}
Let us compute the colored fan of $X=\PP^2$, as in Example~\ref{ex:Pn} with $n=2$. We consider again the set $\D=\{D_3\}$, so $\E=\{D_1, D_2\}$, where $D_i$ is given by the vanishing of the homogeneous coordinate $x_i$. We have $D_1 = X(\alpha_1)=X(\alpha_1+\alpha_2)$, $D_2 = X(\alpha_2)=X(-\alpha_1)$ and $D_3=X(-\alpha_1-\alpha_2) =X(-\alpha_2)$. Then $A\subset\Autz(\PP^2,\D)$ is isomorphic to $\GL(2)$ with roots $\alpha_1$ and $-\alpha_1$; choosing the Borel subgroup of $A$ stabilizing the point $[1,0,0]$ corresponds to choosing $\Psi=\{\alpha_1\}$.  The lattice $\Lambda_A(\PP^2)$ is $\rho_{G,\PP^2}(D_1)^\perp=\ZZ\alpha_2$, and $\PP^2$ has only one $A$-color, namely $D_2$. The maximal colored cones of the original fan $\F_G(\PP^2)$ are the cones $c_{i,j}$ generated by $\rho_{G,\PP^2}(D_i)$ and $\rho_{G,\PP^2}(D_j)$ for all $i< j\in\{1,2,3\}$.

The intersection of $c_{1,3}$ and $c_{2,3}$ with $\Psi^\perp$ is the half line $\QQ_{\geq0}$ if we identify $\N_A(\PP^2)$ with $\QQ$ using the basis dual to $\alpha_2\in\Lambda_A(\PP^2)$. This intersection is a face of both cones, so this yields the maximal colored cone $(\QQ_{\geq0},\varnothing)$ of the fan $\F_A(\PP^2)$. The intersection of $c_{1,2}$ with $\Psi^\perp$ corresponds to $\QQ_{\leq0}$, and it is {\em not} a face of $C_{1,2}$; correspondingly, the unique $A$-color $D_2$ satisfies the condition required to be in $d(c_{1,2})$. This yields the other maximal colored cone of $\F_A(\PP^2)$, namely $(\QQ_{\leq0}, \{D_2\})$.
\end{example}

\section{Semisimple case}\label{s:semisimple}
In this section we assume that $G$ is a semisimple group, i.e.\ $C=\{e\}$. In this setting the functionals associated to the colors of $X$ generate $\N_G(X)$ as a vector space. Indeed, if $\lambda\in\Lambda_G(X)$ is in $\rho_{G,X}(\Delta_G(X))^\perp$, then a rational function $f\in\CC(G/H)^{(B)}_\lambda$ is regular on $G/H$ and nowhere zero. It can be then lifted to a nowhere-vanishing function $F\in\CC[G]$, which is then constant since $G$ has no non-trivial character (see \cite[Proposition~1.2]{KKV89}). We conclude that $\lambda=0$, and the claim follows.

This essentially implies the following main result of this section.

\begin{theorem}\label{thm:semisimple}
If $G$ is semisimple and $\D$ is any subset of $\partial_G X$, then $\Autz(X,\D\cup(\partial_G X)^{\ell})$ is a Levi subgroup of $\Autz(X,\D)$.
\end{theorem}

The proof is at the end of this section. The theorem implies that if $G$ is semisimple then Section~\ref{s:nonlinear} is enough to describe a Levi subgroup of $\Autz(X,\D)$ and its action on $X$, without any restriction on $\D$.

Recall from Section~\ref{s:restricting} the restriction map
\[
\kappa_{x'}\colon (\ker\psi_*)^\circ \to \Autz(X_{x'})
\]
where $x'$ lies on the open $G$-orbit of $X'$, and $X_{x'}=\psi^{-1}(x')$.

\begin{lemma}\label{lemma:verysolvablefiber}
For all $x'$ in the open $B$-orbit of $X'$, the image of $\kappa_{x'}$ in $\Autz(X_{x'})$ is connected and solvable.
\end{lemma}
\begin{proof}
The image $\kappa_{x'}$ is connected since $(\ker\psi_*)^\circ$ is connected; we prove it's solvable. To simplify notations we assume that $x'=x'_0$. Let $\E'\subseteq\partial_S X_{x'_0}$ be the following subset:
\[
\E' = \{ E\cap X_{x'_0} \;|\; E\in \E\},
\]
and define $\D' = \partial_S X_{x'_0}\smallsetminus \E'$. Let us also denote by $K_{x'_0}$ the image of $\kappa_{x'_0}$: it is obviously a subgroup of $\Autz(X_{x'_0},\D')$. On the other hand $K_{x'_0}$ contains the maximal torus $S$ of $\Autz(X_{x'_0})$, hence we only have to compute the root subgroups it contains. Thanks to Lemma~\ref{lemma:automfiber} and Corollary~\ref{cor:dkerpsi}, they are the root spaces $U_\alpha\subset\Autz(X_{x'_0})$ for $\alpha$ varying in the set
\[
R = \left\{ \gamma|_S \;\middle\vert\; 0\neq \gamma\in\Lambda_G(G/H), \; X(\gamma) \textrm{ exists and } X(\gamma) \in \E^\ell \right\}.
\]
From Lemma~\ref{lemma:opposite}, we obtain that $R$ doesn't contain the opposite of any of its elements, therefore $K_{x'_0}$ is solvable.
\end{proof}

\begin{example}
Let $X$ and $G$ be as in Example~\ref{ex:blowup1}. Recall that the $G$-stable prime divisors $D_1$, $D_2$, and $E$ have valuations resp.\ $(-1,1)$, $(1,-1)$ and $(-1,0)$ with respect to the basis given by the two colors, and that the first two are those in $(\partial X)^\ell$. As in the proof of Lemma~\ref{lemma:verysolvablefiber}, there exists no $\gamma\in\Lambda_G(X)$ such that $X(\gamma)=D_1$ and $X(-\gamma)=D_2$.
\end{example}

\begin{proof}[Proof of Theorem~\ref{thm:semisimple}]
First, observe that $(\ker\psi_*)^\circ$ is solvable. This stems from Lemma~\ref{lemma:verysolvablefiber}, and the obvious observation that
\begin{equation}\label{eqn:intker}
\bigcap_{x' \textrm{ in the open $B$-orbit of $X'$}} \ker(\kappa_{x'}) = \left\{\id_{X}\right\}.
\end{equation}

Consider now the variety $X$ under the action of $A=\Autz(X,\D\cup(\partial_G X)^{\ell})$. Thanks to Theorem~\ref{thm:nonlinear}, the group $A$ is semisimple (because here $G$ is semisimple) and under its action $X$ is a $G$-regular embedding with boundary $\D\cup(\partial_G X)^{\ell}$. Corollary~\ref{cor:partialt} implies $(\partial_A X)^{n\ell}=\D^{n\ell}\subseteq \D$, and we deduce that $\Autz(X,\D)\subseteq\Autz(X,(\partial_AX)^{n\ell})$.

Then we may apply Proposition~\ref{prop:kerpsi} with $G$ replaced by the universal cover of $A$: the theorem follows.
\end{proof}

\begin{remark}
Let $X$ and $G$ be as in Example~\ref{ex:blowup1}. Then $\Autz(X)$ is not reductive. Indeed, it must fix the point $p\in\PP^{n+1}$, and it is elementary to conclude that $\Autz(X)$ is the corresponding maximal proper parabolic subgroup of $\PGL(n+2)\times\PGL(n+1)$. The unipotent radical $\Autz(X)^u$ can be studied restricting its elements to the generic fiber $X_{x_0'}$; however, the example shows that for any given fiber the restriction may be non-injective, therefore a global analysis of these restrictions is needed. This goes beyond the scope of the present work.
\end{remark}

\section{Prime divisors on the linear part of the valuation cone}\label{s:linear}
In this section $G=G'\times C$ is neither abelian nor semisimple. Thanks to Section~\ref{s:abelian}, we may assume that $G'$ acts non-trivially on $X$. The variety $X'$ is then not a single point.

Our goal is to study $\Autz(X,\D)$ for a general $\D$. Denote as usual $\E=\partial_G X\smallsetminus \D$.

We begin applying the results of Section~\ref{s:nonlinear} to the group $\widetilde G= \Autz(X,\D\cup(\partial_GX)^\ell)$. From Theorem~\ref{thm:nonlinear} and Corollary~\ref{cor:partialt} it follows that $X$ is a regular $\widetilde G$-variety with boundary $\D\cup(\partial_GX)^\ell$, and the valuations of the elements of $(\partial_GX)^\ell$ lie on the linear part of the valuation cone both with respect to the $G$-action and to the $\widetilde G$-action on $X$.

Therefore we may replace $G$ with the group $\widetilde G$, and proceed to our analysis on $\Autz(X,\D)$ where now $X$ is a complete $G$-regular embedding satisfying $\D\supseteq (\partial_G X)^{n\ell}$.

Recall that $S$ acts on $X$ naturally by $G$-equivariant automorphisms preserving the fibers of $\psi$, so we can consider $S$ as a subgroup of $\Autz(X,\partial_G X)\cap (\ker \psi_*)^\circ$.

\begin{proposition}\label{prop:Levinl}
Let $x'$ in the open $G$-orbit of $X'$, and $L=L(X,\D)$ be a Levi subgroup of $(\Autz(X,\D)\cap\ker(\psi_*))^\circ$ containing $S$. Then $L_{x'}=\kappa_{x'}(L)$ is isomorphic to $L$, and the group
\[
(\theta_{G,X}(G),\theta_{G,X}(G))\times L(X,\D)
\]
is locally isomorphic to a Levi subgroup of $\Autz(X,\D)$.
\end{proposition}
\begin{proof}
Thanks to formula (\ref{eqn:intker}) and the fact that $x'$ is generic in $X'$, we know that the map $L\to L_{x'}$ has unipotent kernel, therefore is an isomorphism. The rest follows from Proposition~\ref{prop:kerpsi}.
\end{proof}

Recall that we have chosen a base point $x_0\in X$ in the open $B$-orbit of $X$, and we denote by $x_0'$ its image in $X'$.

\begin{definition}
We define the following group:
\[
A=A(X,\D)=(\theta_{G,X}(G),\theta_{G,X}(G))\times L(X,\D),
\]
where $L(X,\D)$ is defined as in Proposition~\ref{prop:Levinl} choosing $x'=x_0'$.
\end{definition}

We describe now the reductive group $L_{x_0'}$ in terms of the root subspaces it contains with respect to its maximal torus $S$. 

\begin{definition}
We define
\[
R = R(X,\D) = \left\{ \gamma|_{S} \;\middle\vert\; 0\neq \gamma\in\Lambda_G(G/H),\; X(\gamma) \textrm{ exists and } X(\gamma)\in\E \right\},
\]
and we denote by $\Phi=\Phi(X,\D)$ the maximal subset of $R$ such that $-\alpha\in R$ for every $\alpha\in R$.
\end{definition}

\begin{proposition}\label{prop:Levinldesc}
The set $\Phi(X,\D)$ is a subset of $\Phi(X_{x_0'},\D')$, where $\D'=\{ D\cap X_{x_0'} \;|\; D\in \D^\ell \}$. Moreover, $L_{x_0'}\subseteq \Autz(X_{x_0'})$ is generated by $S$ together with all subgroups $U_{\alpha}$ such that $\alpha\in\Phi(X,\D)$.
\end{proposition}
\begin{proof}
For the first assertion, it is enough for any $\alpha=\gamma|_S\in\Phi(X,\D)$ to restrict the function $f\in H^0(X,\OO_X(X(\gamma)))^{(B)}_\gamma$ to $X_{x_0'}$. Since $S$ is a maximal torus of $\Autz(X_{x_0'})$, the second assertion follows from Lemma~\ref{lemma:automfiber} and Corollary~\ref{cor:dkerpsi}.
\end{proof}

This provides a complete description of the group $A$. It remains now to describe the fan associated to $X$ as an $A$-variety.

Let $0\neq \gamma\in\Lambda(G/H)$ be such that $\gamma|_S=\alpha\in\Phi$, and choose an element $f_\gamma\in H^0(X,\OO_X(X(\gamma)))^{(B)}_\gamma$ such that $f_\gamma(x_0)=1$. Then $\rho_{G,X}(X(\gamma))\in \V_G^\ell(G/H)$ can be considered as an element of $\Hom_\ZZ(\Chi(S), \ZZ)$, and therefore it is canonically associated with a one-parameter subgroup $\mu_\gamma\colon\GG_m\to S$. The torus $S$ acts on $X$ through an identification with a subtorus of $T_{G,X}$, thanks to Lemma~\ref{lemma:Saction}. As in Section~\ref{s:abelian}, we consider the induced tangent vector field on $X$ and we denote it by $\delta_\gamma\in H^0(X,\T_X)$.

\begin{lemma}\label{lemma:xi}
The product $\xi_\gamma=f_\gamma\delta_\gamma$ is a well-defined tangent vector field of $X$, and it is sent to $H^0(X,\OO_X(X(\gamma)))$ via the surjective map of (\ref{eqn:exact}). Its restriction to $X_{x_0'}$ is a tangent vector field and is a generator of the Lie algebra of $U_\alpha\subset \Autz(X_{x_0'})$. Moreover, the action of the one-parameter subgroup of $\Autz(\Zi_{G,X})$ induced by $\xi_\gamma$ can expressed using formula (\ref{eqn:oda-azione}) of Proposition~\ref{prop:oda}.
\end{lemma}
\begin{proof}
Formula (\ref{eqn:oda-azione}) of Proposition~\ref{prop:oda} holds for $\xi_\gamma$ being simply an elementary reformulation of the definition $\xi_\gamma=f_\gamma\delta_\gamma$.

The rational function $f_\gamma$ has its only pole in $X(\gamma)$, which means that we only have to check the first assertion on points of $X(\gamma)$. On $\Zi_{G,X}\cap X(\gamma)$ it can be checked using the fact that $\Zi_{G,X}$ is a toric $T_{G,X}$-variety, and expressing $\xi_\gamma$ in local coordinates. This also implies that $\xi_\gamma$ is a well-defined vector field on $E\cap X_0$, thanks to the $P_{G,X}^u$-invariance of both $f_\gamma$ and $\delta_\gamma$. Then the locus where $\xi_\gamma$ might not be a well-defined vector field has codimension at least $2$, which implies the first statement.

Since $S$ acts on $X$ stabilizing both $\Zi_{G,X}$ and $X_{x_0'}$, we deduce that $\xi_\gamma$ can be restricted to a vector field on both these varieties. The rest follows by expressing $\xi_\gamma$ on $\Zi_{G,X}$ explicitly in local coordinates.
\end{proof}

\begin{definition}
We choose a Borel subgroup $B_A$ of $A$ such that $\theta_{A,X}(B_A)\cap \theta_{G,X}(G)=\theta_{G,X}(B)$ and such that $B_A\cap L$ is a Borel subgroup of $L$. Let us also denote by $\Psi=\Psi(X,\D)\subset\Phi(X,\D)$ the set of simple roots and by $\Phi_+=\Phi_+(X,\D)\subset\Phi(X,\D)$ the set of positive roots associated to the Borel subgroup $B_{L_{x_0'}}=  \kappa_{x_0'}(B_A\cap L)$ of $L_{x_0'}$, with respect to the maximal torus $S$. Finally, let
\[
r\colon \Lambda_G(X)\to\Chi(S)=\Lambda_S(X_{x_0'})
\]
be the restriction of characters of $\Lambda_G(X)$ to $S$ (see Section~\ref{s:relating}).
\end{definition}

We may apply Proposition~\ref{prop:abelian} and Theorem~\ref{thm:abelian} to the toric $S$-variety $X_{x_0'}$ and the sets of roots $\Phi$ and $\Psi$. We obtain a description of $X_{x_0'}$ as an $L_{x_0'}$-variety, and in particular the lattice
\[
\Lambda_{L_{x_0'}}(X_{x_0'}) \subseteq \Lambda_S(X_{x_0'}),
\]
together with the projection
\[
\N_S(X_{x_0'}) \to \N_{L_{x_0'}}(X_{x_0'}).
\]

\begin{proposition}\label{prop:linear}
The restriction of weights from $\theta_{A,X}(B_A)$ to $\theta_{G,X}(B)$ induces an isomorphism \[
\Lambda_A(X) \cong r^{-1}(\Lambda_{L_{x_0'}}(X_{x_0'})) \subseteq \Lambda_G(X).
\]
We denote the corresponding surjective map by
\[
s\colon\N_G(X)\to\N_A(X).
\]
The set of colors of $X$ as a spherical $A$-variety is the following disjoint union:
\[
\Delta_A(X) = \Delta_G(X) \cup \left\{ E \in \E \;\middle\vert\; E\cap X_{x_0'} \in \Delta_{L_{x_0'}}(X_{x'_0})\right\},
\]
and for each $E\in\Delta_A(X)$, we have
\[
\rho_{A,X}(E) = s(\rho_{G,X}(E)).
\]
The simple roots associated to the parabolic subgroup $\P_A(X)$ are the union of the set of those associated to $\P_G(X)$ with the set $\{ \alpha\in \Psi \;|\; X_{x'_0}(-\alpha)\notin \Delta_{L_{x_0'}}(X_{x'_0})\}$.
\end{proposition}
\begin{proof}
A $B_A$-eigenvector in $\CC(X)$ is a fortiori a $B$-eigenvector, thanks to the choice of $B_A$. This induces an inclusion $\Lambda_A(X)\subseteq\Lambda(X)$.

Moreover, a $B$-eigenvector $f\in\CC(X)$ is also a $B_A$-eigenvector if and only if its restriction $f|_{X_{x_0'}}$ is a $B_{L_{x_0'}}$-eigenvector, thanks to the structure of $A$ as described in Proposition~\ref{prop:Levinl}. This proves the first assertion.

Secondly, a color of $X$ as an $A$-variety maps either dominantly onto $X'$, or not. In the first case, its intersection with the (generic) fiber $X_{x_0'}$ is $B_{L_{x_0'}}$-stable but not $L_{x_0'}$-stable (otherwise it would have been $A$-stable).

In the second case, it maps onto a $G$-color of $X'$, i.e. it is a color of $X$ with respect to the $G$ action. The second assertion follows, and implies the last assertion on $\P_A(X)$.
\end{proof}

\begin{example}\label{ex:linear}
Let $G=(\GG_m)^2\times \SL(2)$ act on $Y=\PP^1\times\PP^1\times\PP^1$ as follows: $(\GG_m)^2$ acts only on the first factor $\PP^1=\PP(\CC^2)$ linearly on the two homogeneous coordinates, $\SL(2)$ acts on the second and third factor diagonally. Then consider $X$ to be the projective completion of the line bundle $\OO_{\PP^1}(1)\boxtimes \OO_{\PP^1}(-1)\boxtimes \OO_{\PP^1}$ on $Y$. Then $X$ is a spherical $G$-variety, with open orbit $G/H$ where
\[
H = \left\{ \left(a,b,\left(\begin{array}{cc} a & 0 \\ 0 & a^{-1} \end{array}\right)\right)\;\middle\vert\; a,b\in\CC\smallsetminus\{0\}\right\}.
\]
The variety $X$ is toroidal under the action of $G$, the associated wonderful variety is $\PP^1\times\PP^1$ and the map $X\to\XX$ extending $G/H\to G/\widehat H$ is the projection of $X$ onto $Y$ followed by the projection $Y\to\XX$ on the second and third coordinate. There are two $G$-colors in $X$, namely the inverse images of the two $G$-colors of $\XX$.

There are five $G$-stable prime divisors in $X$, two of them given by the conditions that the projection on the first coordinate of $Y$ is a $(\GG_m)^2$-stable point (i.e.\ $0$ or $\infty$ in the usual coordinates). These two $G$-stable prime divisors $E_1$ and $E_2$ are mapped surjectively onto $\XX$, therefore setting $\D=\partial_G X\smallsetminus\E$ where $\E=\{D_1,D_2\}$ we have $\D\supset (\partial_GX)^{n\ell}$.

A Levi subgroup $A$ of $\Autz(X,\D)$ (actually, the whole group in this case) is the image of $\GL(2)\times \SL(2)$, where we let $\GL(2)$ act on $X$ by lifting its linear action on the first factor $\PP^1$ of $Y$. The $A$-colors of $X$ are the $G$-colors together with one additional prime divisor, namely $D_1$ or $D_2$ (according to the choice of a Borel subgroup of $\GL(2)$).
\end{example}

Let $c$ be a cone of the fan $\F(X)$. Then $c$ is generated as a convex cone by a set of $1$-dimensional faces $F(c)$. We denote by $c^\ell$ the intersection $c\cap \V_G^\ell(X)$, by $F^\ell(c)$ the $1$-dimensional faces of $F(c)$ generating $c^\ell$, and $F^{n\ell}(c) = F(c)\smallsetminus F^\ell(c)$.

Since $c^\ell$ is a cone of the toric $S$-variety $X_{x_0'}$, it corresponds to an $S$-orbit $Y$ on $X_{x_0'}$. As in the proof of Theorem~\ref{thm:abelian}, the corresponding $L_{x_0'}$-orbit $L_{x_0'}Y$ on $X_{x_0'}$ has colored cone $(c^\ell\cap\Psi^\perp,d(c^\ell))$, where the orthogonal $\Psi^\perp$ is taken inside $\V_G^\ell(G/H)$, and $d(c^\ell)$ is a set of $L_{x_0'}$-colors of $X_{x_0'}$.

\begin{definition}
For any $c\in\F(X)$ we define a colored cone $(c_A(c),d_A(c))$, where $c_A(c)\subset \N_A(X)$ and $d_A(c)\subseteq \Delta_A(X)$, as follows. The cone $c_A(c)$ is the convex cone in $\N_A(X)$ generated by $s(F^{n\ell}(c))$ and $s(c^\ell\cap\Psi^\perp)$. The set $d_A(c)$ is the set of colors $E\in\Delta_A(X)$ such that $E\notin \Delta_G(X)$ and $E\cap X_{x_0'}\in d(c^\ell)$.
\end{definition}

\begin{theorem}\label{thm:linear}
The colored fan $\F_A(X)$ of $X$ as a spherical $A$-variety is
\[
\F_A(X) = \left\{ (c_A(c),d_A(c)) \;\middle\vert\; c \in \F_G(X) \right\}.
\]
\end{theorem}
\begin{proof}
Let $Y$ be a $G$-orbit of $X$, with associated cone $c=c_{X,Y}$. We claim that the colored cone associated to the $A$-orbit $AY$ is $(c_A(c),d_A(c))$:  arguing as in the proof of Theorem~\ref{thm:abelian}, this is enough to show the theorem.

To prove the claim, first we show that the set $d'$ of $A$-colors containing $AY$ is equal to $d_A(c)$. Since $X$ is toroidal, no $G$-color contains $Y$, nor $AY$. Therefore any $A$-color $E$ in $d'$ is indeed a $G$-stable prime divisor whose functional lies in $\V_G^\ell(X)$. It intersects $X_{x_0'}$ in an $L_{x_0'}$-color of $X_{x_0'}$, by Proposition~\ref{prop:linear}, and we only have to show that $E\cap X_{x_0'}$ is in $d(c^\ell)$.

We check this fact using the definition of $d(c^\ell)$. Take a positive root $\beta\in\Phi_+$ of $X_{x_0'}$, the prime divisors $X_{x_0'}(\beta)$, $X_{x_0'}(-\beta)$ of $X_{x_0'}$, and suppose that $X_{x_0'}(-\beta)=E\cap X_{x_0'}$, so $\rho_{S,X_{x_0'}}(X_{x_0'}(-\beta))$ lies on a $1$-codimensional face of $c^\ell$. We have to show that $\rho_{S,X_{x_0'}}(X_{x_0'}(\beta))$ also lies on a $1$-codimensional face of $c^\ell$, in other words that $X_{x_0'}(\beta)$ contains the $S$-orbit of $X_{x_0'}$ associated $c^\ell$.

Now $E=E_1$ and some other element $E_2\in\E$ satisfy $E_1\cap X_{x_0'} = X_{x_0'}(-\beta)$, $E_2\cap X_{x_0'} = X_{x_0'}(\beta)$, and $-\beta$ and $\beta$ are the restrictions to $S$ of resp.\ $\gamma_1,\gamma_2\in\Lambda_G(X)$, such that $X(\gamma_i)=E_i$ for $i=1,2$. Suppose that $E_2$ doesn't contain $Y$. Then we consider $\Zi_{G,X}$: intersecting it with $E_1$, $E_2$ and $Y$ two $T_{G,X}$-stable prime divisors and a $T_{G,X}$-orbit, such that $E_1\cap\Zi_{G,X}\supseteq Y\cap\Zi_{G,X}$ and $E_2\cap\Zi_{G,X}\nsupseteq Y\cap\Zi_{G,X}$.

At this point we follow the same approach of the proof of Theorem~\ref{thm:abelian}, statement (\ref{proof:abelian:colors}), applied to the toric variety $\Zi_{G,X}$ and the automorphisms induced by the tangent vector field $\xi_{\gamma_1}$ (as defined in Lemma~\ref{lemma:xi}). This yields the formula (\ref{eqn:color1}) for $\xi_{\gamma_1}$, which shows that $E_1\cap\Zi_{G,X}$ doesn't contain $AY\cap\Zi_{G,X}$: a contradiction. As a consequence $E_2\supseteq Y$, so $X_{x_0'}(\beta)$ contains the $S$-orbit of $X_{x_0'}$ associated $c^\ell$. This concludes the proof of the inclusion $d'\subseteq d_A(c)$.

Let now $D\in d_A(c)$. Then, by Theorem~\ref{thm:abelian}, the intersection $D\cap X_{x_0'}$ contains the $L_{x_0'}$-orbit of $X_{x_0'}$ corresponding to $(c^\ell\cap \Psi^\perp, d(c^\ell))$. Let $y$ be a point on this orbit: then $D$ contains $\overline{Ay}$.

On the other hand, from the proof of Theorem~\ref{thm:abelian}, we see that $\overline{L_{x_0'}y}$ contains the $S$-orbit of $X_{x_0'}$ corresponding to $c^\ell\subset\N_S(X_{x_0'})$. It follows that $\overline{Ay}$ contains the $G$-orbit of $X$ associated to $c^\ell\subset \N_G(X)$, and thus also the $G$-orbit $Y$ associated to $c\subset \N_G(X)$. Being $A$-stable, $\overline{Ay}$ must then contain $AY$ too, and since $D$ is closed, we obtain $D\supseteq AY$. I.e., $D$ is in $d'$.

We now prove that the convex cone $c'$ associated to $AY$ is $c_A(c)$. First observe that $Y$ and $AY$ are contained in the same elements of $(\partial_G X)^{n\ell}$, since $L$ stabilizes all fibers of $\psi$. Therefore $c'$ is generated by $s(F^{n\ell}(c))$ and its intersection with $s(\V_G^\ell(X))$. It remains to prove that $c'\cap s(\V_G^\ell(X))=s(c^\ell\cap \Psi^\perp)$.

The cone $c'\cap s(\V_G^\ell(X))$ is generated by $\rho_{A,X}(E)$ where $E\in(\partial_G X)^\ell$ is:
\begin{itemize}
\item[(1)] an $A$-color of $X$ containing $AY$, i.e.\ $E\in d_A(c)$, or
\item[(2)] an $A$-stable prime divisor containing $AY$.
\end{itemize}
On the other hand the generators of $s(c^\ell\cap\Psi^\perp)$ are the elements $\rho_{A,X}(E)$ where $E\in(\partial_G X)^\ell$ is:
\begin{itemize}
\item[(1')] an $A$-color such that $E\cap X_{x_0'}$ is a color containing the $L_{x_0'}$-orbit $Z$ of $X_{x_0'}$ associated to $(c^\ell\cap\Psi^\perp,d(c^\ell))$, or
\item[(2')] an $A$-stable prime divisor such that $E\cap X_{x_0'}$ is a $L_{x_0'}$-stable prime divisor containing $Z$.
\end{itemize}
Thanks to the first part of the proof, the prime divisors $E$ of type (1) and of type (1') are the same.

If $E$ is of type (2') then it contains $AZ$, whose closure in turn contains $AY$. Therefore $E$ is of type (2). Let now $E$ be of type (2). Then $E\cap \Zi_{G,X}$ is an $L$-stable (and $T_{G,X}$-stable) prime divisor of $\Zi_{G,X}$ containing $Y\cap\Zi_{G,X}$, which is the $T_{G,X}$-orbit of $\Zi_{G,X}$ associated with $c$, and $\rho_{T_{G,X},\Zi_{G,X}}(E\cap \Zi_{G,X})$ lies on $\V_G^\ell(G/H)$. Hence $E\cap X_{x_0'}$ is an $L_{x_0'}$-stable prime divisor of $X_{x_0'}$ containing the $S$-orbit of $X_{x_0'}$ associated with $c^\ell$. Thanks to the proof of Theorem~\ref{thm:abelian}, we deduce that $E\cap X_{x_0'}$ contains $Z$, i.e.\ $E$ is of type (2').
\end{proof}

\begin{corollary}\label{cor:horosphericallinear}
$\Sigma_A(X)=\Sigma_G(X)$.
\end{corollary}
\begin{proof}
The proof is similar to the proof of Corollary~\ref{cor:abelianSigma}.
\end{proof}

\begin{remark}
Although a Levi subgroup $A$ of $\Autz(X,\D)$ splits (up to central isogeny) into a product $(\theta_{G,X}(G),\theta_{G,X}(G))\times L(X,\D)$, Example~\ref{ex:linear} shows that the open $A$-orbit of $X$ in general is not the product of a $(G,G)$-homogeneous space and an $L(X,\D)$-homogeneous space. On the other hand, from Corollary~\ref{cor:horosphericallinear} it follows that the factor $L(X,\D)$ acts trivially on $\XX$, whose open $A$-orbit is then always a homogeneous space under the action of $(G,G)$.
\end{remark}


\begin{thebibliography}{64}









\bibitem[BB96]{BB96} F.\ Bien, M.\ Brion: \textit{Automorphisms and local rigidity of regular varieties}, Compositio Math.\ {\bf 104} (1996), no.\ 1, 1--26. 

\bibitem[BDP90]{BDP90} E.~Bifet, C.~De Concini, C.~Procesi: \textit{Cohomology of regular embeddings}, Adv.\ in Math.\ {\bf 82} (1990), 1--34. 

\bibitem[BL11]{BL11} P.~Bravi, D.~Luna: \textit{An introduction to wonderful varieties with many examples of type $\mathsf F_4$}, J.\ Algebra {\bf 329} (2011), 4--51. 


\bibitem[Br89]{Br89} M.~Brion: \textit{Groupe de Picard et nombres caract\'eristiques des vari\'et\'es sph\'eriques}. Duke Math. J.  {\bf 58}  (1989),  no. 2, 397--424.

\bibitem[Br90]{Br90} M.~Brion: \textit{Vers une g\'en\'eralisation des espaces sym\'etriques}, J.\ Algebra \textbf{134} (1990), no.\ 1, 115--143.


\bibitem[Br97]{Br97} M.~Brion: \textit{Vari\'et\'es sph\'eriques},\\ \verb|http://www-fourier.ujf-grenoble.fr/~mbrion/spheriques.ps|

\bibitem[Br07]{Br07} M.~Brion:  \textit{The total coordinate ring of a wonderful variety}, J.\ Algebra {\bf 313} (2007), no.\ 1, 61--99.


\bibitem[BP87]{BP87} Brion, F.~Pauer. Valuations des espaces homog\`enes sph\'eriques. \textit{Commentarii Mathematici Helvetici} \textbf{62} (1987), no. 2, 265--285.






\bibitem[De70]{De70} M.\ Demazure, \textit{Sous-groupes alg\'ebriques de rang maximum du groupe de Cremona},  Ann.\ Sci.\ \'Ecole Norm.\ Sup.\  {\bf 3}, (1970), no.\ 4, 507--588.

\bibitem[De77]{De77} M.\ Demazure, \textit{Automorphismes et d\'eformations des vari\'et\'es de Borel},  Invent.\ Math.\  {\bf 39}  (1977), no. 2, 179--186.

\bibitem[Gi89]{Gi89} V.~Ginzburg: \textit{Admissible modules on a symmetric space}, Ast\'erisque {\bf 173}--{\bf 174} (1989), 199--255.







\bibitem[Kn91]{Kn91} F.~Knop: \textit{The Luna-Vust theory of spherical embeddings}, Proceedings of the Hyderabad Conference on Algebraic Groups (Hyderabad, 1989), 225--249, Manoj Prakashan, Madras, 1991.

\bibitem[Kn94]{Kn94} F.~Knop: \textit{The assymptotic behaviour of invariant collective motion}, Invent.\ Math. {\bf 114}, 1994, 309--328.

\bibitem[Kn96]{Kn96} F. Knop: \textit{Automorphisms, root systems, and compactifications of homogeneous varieties}, J.\ Amer.\ Math.\ Soc.\ \textbf{9} (1996), no.\ 1, 153--174.

\bibitem[KKLV89]{KKLV89} F.~Knop, H.~Kraft, D.~Luna, T.~Vust: \textit{Local properties of algebraic group actions}, in \textit{Algebraische Transformationsgruppen und Invariantentheorie} (H. Kraft, P. Slodowy, T. Springer eds.) DMV-Seminar 13, Birkh\"auser Verlag (Basel-Boston) (1989) 63-76.

\bibitem[KKV89]{KKV89} F.~Knop, H.~Kraft, T.~Vust: \textit{The Picard group of a G-variety},  in \textit{Algebraische Transformationsgruppen und Invariantentheorie} (H. Kraft, P. Slodowy, T. Springer eds.) DMV-Seminar 13, Birkh\"auser Verlag (Basel-Boston) (1989) 77-87.





\bibitem[Lu01]{Lu01} D.~Luna: \textit{Vari\'et\'es sph\'eriques de type $A$}, Inst.\ Hautes \'Etudes Sci.\ Publ.\ Math.\ \textbf{94} (2001), 161--226.




\bibitem[Oda88]{Oda88} T.~Oda: \textit{Convex Bodies and Algebraic Geometry: An Introduction to Toric Varieties}, Springer-Verlag, New York-Berlin-Heidelberg, 1988.




\bibitem[Pe09]{Pe09} G.~Pezzini: \textit{Automorphisms of wonderful varieties}, Transform.\ Groups \textbf{14} (2009), no.\ 3, 677--694.




\bibitem[Wa96]{Wa96} B.~Wasserman: \textit{Wonderful varieties of rank two}, Transform.\ Groups \textbf{1} (1996), no.\ 4, 375--403.



\end{thebibliography}
\end{document}